\documentclass[10pt,british,a4paper,reqno]{amsart}
\usepackage[T1]{fontenc}
\usepackage[utf8]{inputenc}
\usepackage[british]{babel}

	\usepackage{amsmath}
	\usepackage{amssymb}
	\usepackage{amsthm}
	\usepackage{bbm}
	\usepackage{braket}
	\usepackage{xcolor}
	\usepackage[shortlabels]{enumitem}
	\usepackage{fancyhdr}
	\usepackage{graphicx}
	\usepackage{ifsym}
	\usepackage{indentfirst}
	\usepackage{lmodern}
	\usepackage{mathrsfs}
	\usepackage{mathtools}
	\usepackage{proof}
	\usepackage{qtree}
	\usepackage{scalerel}
	\usepackage{setspace}
	\usepackage{tensor}
	\usepackage{tikz}
	\usepackage{tikz-cd}
	\usepackage[hyperfootnotes=false]{hyperref}
		\usepackage{bookmark}
	\usepackage{geometry}

	\bookmarksetup{numbered,open}

	\renewcommand{\geq}{\geqslant}
	\renewcommand{\leq}{\leqslant}
	\renewcommand{\phi}{\varphi}

	\providecommand{\corollaryname}{Corollary}
	\providecommand{\definitionname}{Definition}
	\providecommand{\examplename}{Example}
	\providecommand{\lemmaname}{Lemma}
	\providecommand{\propositionname}{Proposition}
	\providecommand{\remarkname}{Remark}
	\providecommand{\theoremname}{Theorem}
	\providecommand{\setupname}{Setup}
	\providecommand{\conjecturename}{Conjecture}
	\providecommand{\questionname}{Question}
	\providecommand{\claimname}{Claim}

	\theoremstyle{plain}
		\newtheorem{thm}{\protect\theoremname}[section] 
		
		\newtheorem{prop}[thm]{\protect\propositionname}
		\newtheorem{lem}[thm]{\protect\lemmaname}
		\newtheorem{cor}[thm]{\protect\corollaryname}

	\theoremstyle{definition}
	 	
		\newtheorem{defn}[thm]{\protect\definitionname}
		\newtheorem{example}[thm]{\protect\examplename}
		\newtheorem{setup}[thm]{\setupname}

	\theoremstyle{remark}
		\newtheorem{rem}[thm]{\protect\remarkname}
		
	\numberwithin{figure}{section}
	\numberwithin{equation}{section}

	\makeatletter
	
		\newcommand\ackname{Acknowledgements}
		\newenvironment{acknowledgements}{%
			\medskip
			\bgroup
			\list{}{\labelwidth\z@
			\leftmargin3pc \rightmargin\leftmargin
			\listparindent\normalparindent \itemindent\z@
			\parsep\z@ \@plus\p@
			
				}%
				\Small
				\item[\hskip\labelsep\scshape\ackname.]%
			}{%
			\endlist\egroup
			}
			
	\makeatother

	\usetikzlibrary{matrix,arrows,decorations.pathmorphing,positioning,decorations.pathreplacing}
	\tikzset{commutative diagrams/.cd, 
		mysymbol/.style = {start anchor=center, end anchor = center, draw = none}}
	\tikzcdset{every label/.append style = {font = \footnotesize}}

	\newcommand{\commutes}[2][\circlearrowleft]{\arrow[mysymbol]{#2}[description]{#1}}

	\newcommand{\BE}{\mathbb{E}}

	\newcommand{\BZ}{\mathbb{Z}}

	\newcommand{\CA}{\mathcal{A}}
	\newcommand{\CB}{\mathcal{B}}
	\newcommand{\CC}{\mathcal{C}}

	\newcommand{\CT}{\mathcal{T}}

	\newcommand{\CX}{\mathcal{X}}

	\newcommand{\SF}{\mathscr{F}}
	\newcommand{\SG}{\mathscr{G}}

		\newcommand{\Ab}{\operatorname{\mathsf{Ab}}\nolimits}

		\newcommand{\skel}{\mathrm{skel}}
		\newcommand{\op}{\mathrm{op}}

		\newcommand{\iso}{\cong}

		\newcommand{\Hom}{\operatorname{Hom}\nolimits}

		\newcommand{\into}{\hookrightarrow}
		\newcommand{\onto}{\rightarrow\mathrel{\mkern-14mu}\rightarrow}

		\newcommand{\idfunc}[1]{\mathbbm{1}_{#1}}

		\newcommand{\fs}{\mathfrak{s}}
		
		\newcommand{\sus}{\Sigma} 

		\newcommand{\lmod}[1]{{#1}\operatorname{--\,\mathsf{mod}}\nolimits}

		\newcommand{\der}{\mathsf{D}}
		\newcommand{\kom}{\mathsf{K}}
		\newcommand{\com}{\mathsf{C}}

		\newcommand{\cone}{\mathrm{MC}}

	\newcommand{\deff}{\coloneqq}
	\newcommand{\eps}{\varepsilon}
	\newcommand{\lan}{\langle}
	\newcommand{\ran}{\rangle}

	\newcommand{\wh}[1]{\widehat{#1}}

	\newcommand\restr[2]{{\left.\kern-\nulldelimiterspace#1
						\right|_{#2}}}

	\makeatletter
		\@namedef{subjclassname@2020}{\textup{2020} Mathematics Subject Classification}
		
	\makeatother

\begin{document}

\title{Transport of structure in higher homological algebra}
    \author[Bennett-Tennenhaus]{Raphael Bennett-Tennenhaus}
        \address{Faculty of Mathematics\\
        Bielefeld University\\
        Universit{\"{a}}tsstra{\upshape\fontsize{6pt}{6pt}\selectfont ß}e 25\\
        33615 Bielefeld\\
        Germany}
        \email{raphaelbennetttennenhaus@gmail.com}

    \author[Shah]{Amit Shah}
        \address{School of Mathematics\\ 
        University of Leeds\\ 
        Leeds, LS2 9JT\\ 
        United Kingdom}
        \email{a.shah1@leeds.ac.uk}

\date{\today}




\begin{abstract}

We fill a gap in the literature regarding `transport of structure' for $(n+2)$-angulated, $n$-exact, $n$-abelian and $n$-exangulated categories appearing in (classical and higher) homological algebra. 
As an application of our main results, we show that a skeleton of one of these kinds of categories inherits the same structure in a canonical way, up to equivalence. 
In particular, it follows that a skeleton of a weak $(n+2)$-angulated category is in fact what we call a strong $(n+2)$-angulated category. When $n=1$ this clarifies a technical concern with the definition of a cluster category. We also introduce the notion of an $n$-exangulated functor between $n$-exangulated categories. 
This recovers the definition of an $(n+2)$-angulated functor when the categories concerned are $(n+2)$-angulated, and the higher analogue of an exact functor when the categories concerned are $n$-exact.

\end{abstract}

\maketitle


\section{Introduction}
`Transport of structure' refers to the situation in which an object gains structure by being isomorphic to an object with some pre-existing structure (see \cite[\S{}IV.1]{Bourbaki-elements-of-mathematics-theory-of-sets}). 
Instances of this well-known idea appear, for example, in: 
Lie theory 
\cite{Arnal-analytic-vectors-and-irreducible-representations-of-nilpotent-lie-groups-and-algebras}, 
\cite{Lusztig-cuspidal-local-systems-and-graded-hecke-algebras-I}, 
\cite{Rentschler-orbites-dans-le-spectre-primitif-de-algebre-enveloppante-une-algebre-de-Lie};
operational calculus  \cite{Buchmann-examples-of-convolution-products}; 
algebraic geometry and number theory \cite{BruynReichstein-smoothness-in-algebraic-geography}, \cite{HambletonLemmermeyer-arithmetic-of-pell-surfaces}, \cite{ManasaShankar-pell-surfaces-and-elliptic-curves}; 
the theory of concrete geometrical categories \cite{Diers-classification-of-concrete-geometrical-categories}; 
modular representation theory \cite{Yaraneri-a-filtration-of-the-modular-representation-functor}; and the theory of algebraic structures \cite{Holm-a-note-on-transport-of-algebraic-structures}. These examples involve transporting some algebraic structure across a function between sets.

Analogously, one can ask which categorical structures can be transported across equivalences. 
For example, it is well-known that any category equivalent to an additive category is also additive (see Theorem \ref{thm:transport-of-additive-structure}). Other examples found in the literature include the transport of monoidal \cite{DayStreet-Quantum-categories-star-autonomy-and-quantum-groupoids}, model \cite{Fresse-modules-over-operads-and-functors} and triangulated \cite{Chen-the-singularity-category-of-an-algebra-with-radical-square-zero} structures. 
See also 
\cite{Gelinas-contributions-to-the-stable-derived-categories-of-gorenstein-rings}, 
\cite{KellyLack-monoidal-functors-generated-by-adjunctions-with-applications-to-transport-of-structure}, 
\cite{Kunzer-thesis}.

The motivation for writing this note arose in the context of triangulated categories. 
Suppose $(\CC,\sus, \CT)$ is a triangulated category, where the suspension $\sus$ is assumed to be an automorphism of $\CC$; see Example \ref{exa:triangulated-category-as-special-case}. 
Let $\CC'$ be any category and assume $\SF\colon \CC\to \CC'$ is an equivalence. 
Suppose one wants to find an automorphism $\Sigma'$ of $\CC'$ (and a collection $\CT'$ of triangles in $\CC'$), such that $(\CC,\Sigma',\CT')$ is a triangulated
category and that $\SF$ is a triangle equivalence. 
A canonical choice for $\sus'$ is $\SF\sus\SG$, where $\SG\colon \CC'\to\CC$ is a quasi-inverse for $\SF$. Unfortunately, when one tries to prove this choice is an automorphism of $\CC'$, one encounters problems that seem difficult (or, perhaps, impossible) to resolve. Bernhard Keller pointed out that it is more natural to ask only that $\sus$ is an autoequivalence in the definition of a triangulated category, and we call these `weak' triangulated categories here in order to make a distinction; see Definition \ref{def:weak-pre-and-n-angulated-category}. 
After this modification, by which nothing is lost (see \cite[\S{}2]{KellerVossieck-sous-les-categories-derivees}), we can obtain a clean `transport of triangulated structure' statement. 
Although well-known, we could not find the details of this claim. 
Theorem \ref{thm:transport-of-n-angulated-structure} fills this gap in the literature.

Recently, homological algebra has seen the introduction of \emph{higher (dimensional)} analogues of categories that have been a core part of the classical theory for many years. 
Let $n\geq 1$ be a positive integer. 
The notion of an \emph{$(n+2)$-angulated} category (see Definition \ref{def:pre-and-n-angulated-category}) was introduced in \cite{GeissKellerOppermann-n-angulated-categories}, in a way so that one recovers the definition of a triangulated category by setting $n=1$. 
Similarly, higher versions of exact and abelian categories were introduced in \cite{Jasso-n-abelian-and-n-exact-categories}; see Definitions \ref{def:n-exact-category} and \ref{def:n-abelian-category}. 
Parallel to this, Nakaoka and Palu defined an \emph{extriangulated} category, which is a simultaneous generalisation of a triangulated category and an exact category; see \cite{NakaokaPalu-extriangulated-categories-hovey-twin-cotorsion-pairs-and-model-structures}. 
And, naturally, a higher analogue, called an \emph{$n$-exangulated} category, of an extriangulated category has since been introduced in \cite{HerschendLiuNakaoka-n-exangulated-categories-I-definitions-and-fundamental-properties}; see Definition \ref{def:n-exangulated-category}.

As a consequence of these developments, we include transport of structure results for $(n+2)$-angulated, $n$-exact, $n$-abelian and $n$-exangulated categories. That is, we show that one can transport these categorical structures across equivalences; see Theorems \ref{thm:transport-of-n-angulated-structure}, \ref{thm:transport-of-n-exact-structure}, \ref{thm:transport-of-n-abelian-structure} and \ref{thm:transport-of-n-exangulated-structure}. 
Moreover, in order to state the latter three results in the same fashion as Theorem \ref{thm:transport-of-n-angulated-structure}, we introduce the notion of an \emph{$n$-exact} functor (see Definition \ref{def:n-exact-functor}) and an \emph{$n$-exangulated} functor (see Definition \ref{def:n-exangulated-functor}). 
We show in Theorem \ref{thm:n+2-angulated-functor-iff-n-exangulated} that a (covariant) functor between $(n+2)$-angulated categories is an $n$-exangulated functor if and only if it is an \emph{$(n+2)$-angulated} functor, in the sense of Definition \ref{def:n-angulated-functor}. 
Similarly, in Theorem \ref{thm:n-exact-functor-iff-n-exangulated} we show that a functor between $n$-exact categories is an $n$-exangulated functor if and only if it is an $n$-exact functor.

An application of Theorem \ref{thm:transport-of-n-exact-structure} 
(respectively, Theorem \ref{thm:transport-of-n-abelian-structure}, Theorem \ref{thm:transport-of-n-exangulated-structure}) shows that any skeleton of an $n$-exact (respectively, $n$-abelian, $n$-exangulated) category is again $n$-exact (respectively, $n$-abelian, $n$-exangulated). 
Furthermore, in the $(n+2)$-angulated case we deduce a stronger conclusion: we show that each skeleton of a \emph{weak} $(n+2)$-angulated category (see Definition \ref{def:weak-pre-and-n-angulated-category}) is what we call a \emph{strong} $(n+2)$-angulated category (i.e. with an $(n+2)$-suspension functor that is an automorphism); see Corollary \ref{cor:transport-of-n-angulated-structure-to-skeleton}. 
In this case, the natural choice for the $(n+2)$-suspension functor of the skeleton is an automorphism; see Proposition \ref{prop:equivalence-between-skeletal-categories-is-an-isomorphism}.

This paper is organised as follows. In \S\ref{sec:higher-structures} we recall the definitions of the categories from higher homological algebra that we work with in the sequel. 
In particular, we propose a definition for an $n$-exact functor in \S\ref{sec:n-exact-n-abelian-categories} and an $n$-exangulated functor in \S\ref{sec:n-exangulated-categories}.
Moreover, we show that an $n$-exangulated functor simultaneously generalises the notions of an $(n+2)$-angulated functor and an $n$-exact functor. 
In \S\ref{sec:transport-of-structure} we prove our main results on transport of structure. In \S\ref{sec:special-case-skeletal-categories} we provide an improvement of transport of $(n+2)$-angulated structure when transporting the structure to a skeletal category. 
We also explain how our results from \S\ref{sec:special-case-skeletal-categories} may be used to show how the definition of a cluster category is made rigorous.


\section{Higher structures}
\label{sec:higher-structures}

Throughout \S\ref{sec:higher-structures}, let $n\geq 1$ be a positive integer and let $\CC$ be an additive category.
In \S\S\ref{sec:n-angulated-categories}--\ref{sec:n-exact-n-abelian-categories} we recall the higher analogues of the classical (i.e. $n=1$) theory, and in \S\ref{sec:n-exangulated-categories} we recall the more recent theory of $n$-exangulated categories.


\subsection{\texorpdfstring{$(n+2)$}{n+2}-angulated categories}\label{sec:n-angulated-categories}

We assume throughout \S\ref{sec:n-angulated-categories} that $\CC$ is equipped with an autoequivalence $\sus\colon\CC \to \CC$. 
We recall the definition of an $(n+2)$-angulated category (see Definition \ref{def:pre-and-n-angulated-category}), which was introduced in \cite{GeissKellerOppermann-n-angulated-categories} and we follow the exposition therein.

\begin{rem}

Definitions \ref{def:n+2-sus-sequences}--\ref{def:weak-pre-and-n-angulated-category} below are essentially \cite[Def.\ 2.1]{GeissKellerOppermann-n-angulated-categories}. However, there is one key difference. In \cite{GeissKellerOppermann-n-angulated-categories} the endofunctor $\sus\colon\CC\to\CC$ is assumed to be an automorphism, whereas here we only assume that $\sus$ is an autoequivalence. 

\end{rem}

\begin{defn}
\label{def:n+2-sus-sequences}

A sequence of morphisms of the form
$\begin{tikzcd}[column sep=0.7cm]
X^{0}\arrow{r}{d_{X}^{0}}&X^{1}\arrow{r}{d_{X}^{1}}&\cdots\arrow{r}{d_{X}^{n}}&X^{n+1}\arrow{r}{d_{X}^{n+1}}&\sus X^{0}\end{tikzcd}$ in $\CC$ is called an \emph{$(n+2)$-$\sus$-sequence}.

\end{defn}

\begin{defn}
\label{def:n+2-sus-morphisms}

A \emph{morphism of $(n+2)$-$\sus$-sequences} from 
\[
\begin{tikzcd}[column sep=0.7cm]X^{0}\arrow{r}{d_{X}^{0}}&X^{1}\arrow{r}{d_{X}^{1}}&\cdots\arrow{r}{d_{X}^{n}}&X^{n+1}\arrow{r}{d_{X}^{n+1}}&\sus X^{0}\end{tikzcd}
\]
to 
\[
\begin{tikzcd}[column sep=0.7cm]Y^{0}\arrow{r}{d_{Y}^{0}}&Y^{1}\arrow{r}{d_{Y}^{1}}&\cdots\arrow{r}{d_{Y}^{n}}&Y^{n+1}\arrow{r}{d_{Y}^{n+1}}&\sus Y^{0}\end{tikzcd}
\]
is a tuple $(f^{0},\ldots, f^{n+1})$ of morphisms $f^{i}\colon X^{i} \to Y^{i}$ in $\CC$ such that 
\[
\begin{tikzcd}[column sep=0.7cm]
X^{0}\arrow{r}{d_{X}^{0}}\arrow{d}{f^{0}}&X^{1}\arrow{r}{d_{X}^{1}}\arrow{d}{f^{1}}&\cdots\arrow{r}{d_{X}^{n}}&X^{n+1}\arrow{r}{d_{X}^{n+1}}\arrow{d}{f^{n+1}}&\sus X^{0}\arrow{d}{\sus f^{0}}\\
Y^{0}\arrow{r}{d_{Y}^{0}}&Y^{1}\arrow{r}{d_{Y}^{1}}&\cdots\arrow{r}{d_{Y}^{n}}&Y^{n+1}\arrow{r}{d_{Y}^{n+1}}&\sus Y^{0}
\end{tikzcd}
\]
commutes in $\CC$.

\end{defn}

\begin{defn}
\label{def:weak-pre-and-n-angulated-category}

Recall that $n\geq 1$ is a positive integer and $\CC$ is an additive category equipped with an autoequivalence $\sus$. 
Let $\CT$ be a class of $(n+2)$-$\sus$-sequences in $\CC$. Then we call $(\CC,\sus,\CT)$ a \emph{weak pre-$(n+2)$-angulated category} if the following axioms are satisfied.
\begin{enumerate}[(F1)]

    \item \label{F1}
        \begin{enumerate}[(a)]
        
            \item The class $\CT$ is closed under finite taking direct sums and direct summands.
            \item For each $X\in\CC$, the $(n+2)$-$\sus$-sequence
            $\begin{tikzcd}[column sep=0.6cm]
            X \arrow{r}{1_{X}}& X \arrow{r}{}& 0\arrow{r}{} & \cdots \arrow{r}{}& 0\arrow{r}{} & \sus X
            \end{tikzcd}$ is in $\CT$.
            \item For each morphism $d_{X}^{0}$  there is an $(n+2)$-$\sus$-sequence in $\CT$ whose first morphism is $d_{X}^{0}$.
            
        \end{enumerate}
    \item \label{F2}An $(n+2)$-$\sus$-sequence $\begin{tikzcd}[column sep=0.7cm]X^{0}\arrow{r}{d_{X}^{0}}&X^{1}\arrow{r}{d_{X}^{1}}&\cdots\arrow{r}{d_{X}^{n}}&X^{n+1}\arrow{r}{d_{X}^{n+1}}&\sus X^{0}\end{tikzcd}$ lies in $\CT$ if and only if its \emph{left rotation} 
    \begin{center}
    $
    \begin{tikzcd}[column sep=1.6cm] X^{1}\arrow{r}{d_{X}^{1}}&X^{2}\arrow{r}{d_{X}^{2}}&\cdots\arrow{r}{d_{X}^{n}}&X^{n+1}\arrow{r}{d_{X}^{n+1}}&\sus X^{0}\arrow{r}{(-1)^{n}\sus d_{X}^{0}}&\sus X^{1}\end{tikzcd}
    $
    \end{center}
    lies in $\CT$.
    \item \label{F3}Each commutative diagram 
    \begin{equation}\label{eqn:F3-diagram-1}
    \begin{tikzcd}[column sep=0.7cm]
    X^{0}\arrow{r}{d_{X}^{0}}\arrow{d}{f^{0}}&X^{1}\arrow{r}{d_{X}^{1}}\arrow{d}{f^{1}}&X^{2}\arrow{r}{d_{X}^{2}}&\cdots\arrow{r}{d_{X}^{n}}&X^{n+1}\arrow{r}{d_{X}^{n+1}}&\sus X^{0}\arrow{d}{\sus f^{0}}\\
    Y^{0}\arrow{r}{d_{Y}^{0}}&Y^{1}\arrow{r}{d_{Y}^{1}}&Y^{2}\arrow{r}{d_{Y}^{2}}&\cdots\arrow{r}{d_{Y}^{n}}&Y^{n+1}\arrow{r}{d_{Y}^{n+1}}&\sus Y^{0}
    \end{tikzcd}
    \end{equation}
    with rows in $\CT$ can be completed to a morphism
    \begin{equation}\label{eqn:F3-diagram-2}
    \begin{tikzcd}[column sep=0.7cm]
    X^{0}\arrow{r}{d_{X}^{0}}\arrow{d}{f^{0}}&X^{1}\arrow{r}{d_{X}^{1}}\arrow{d}{f^{1}}&X^{2}\arrow{r}{d_{X}^{2}}\arrow[dotted]{d}{f^{2}}&\cdots\arrow{r}{d_{X}^{n}}&X^{n+1}\arrow{r}{d_{X}^{n+1}}\arrow[dotted]{d}{f^{n+1}}&\sus X^{0}\arrow{d}{\sus f^{0}}\\
    Y^{0}\arrow{r}{d_{Y}^{0}}&Y^{1}\arrow{r}{d_{Y}^{1}}&Y^{2}\arrow{r}{d_{Y}^{2}}&\cdots\arrow{r}{d_{Y}^{n}}&Y^{n+1}\arrow{r}{d_{Y}^{n+1}}&\sus Y^{0}
    \end{tikzcd}
    \end{equation}
    of $(n+2)$-$\sus$-sequences.
    
\end{enumerate}

In this case, we say: $\CT$ is a \emph{pre-$(n+2)$-angulation}; the members of $\CT$ are \emph{$(n+2)$-angles}; and $\Sigma$ is an \emph{$(n+2)$-suspension}. 

Moreover, if $(\CC,\sus,\CT)$ satisfies \ref{F1}--\ref{F3} and \ref{F4} below, then we call $(\CC,\sus,\CT)$ a \emph{weak $(n+2)$-angulated category}. In this case, we refer to $\CT$ as an \emph{$(n+2)$-angulation}.
\begin{enumerate}[(F1)]

\setcounter{enumi}{3}
    \item\label{F4} In the situation of \ref{F3} above, \eqref{eqn:F3-diagram-1} can be completed to 
    \eqref{eqn:F3-diagram-2} using morphisms $f^{2},\ldots,f^{n+1}$ such that the \emph{cone} $C^{\raisebox{0.5pt}{\scalebox{0.6}{$\bullet$}}}=C(f)^{\raisebox{0.5pt}{\scalebox{0.6}{$\bullet$}}}$ lies in $\CT$, where $C^{\raisebox{0.5pt}{\scalebox{0.6}{$\bullet$}}}$ is the $(n+2)$-$\sus$-sequence
	\begin{center}
    	$
	\begin{tikzcd}
X^{1}\oplus Y^{0} \arrow{r}{d_{C}^{0}}&X^{2}\oplus Y^{1}\arrow{r}{d_{C}^{1}}& \cdots\arrow{r}{d_{C}^{n}}&\sus X^{0}\oplus Y^{n+1} \arrow{r}{d_{C}^{n+1}}&\sus X^{1}\oplus \sus Y^{0},
	\end{tikzcd}
	$
	\end{center}
in which 
	\begin{center}
    	$
d_{C}^{i}\deff
\begin{pmatrix}-d_{X}^{i+1} & 0\\ f^{i+1} & d_{Y}^{i} \end{pmatrix},
    	$
	\end{center}
for $i\in\{0,\ldots,n+1\}$, where  $d_{X}^{n+2}\deff \sus d_{X}^{0}$ and $f^{n+2}\deff \sus f^{0}$.

\end{enumerate}

\end{defn}

The definition of a (pre-)$(n+2)$-angulated category, as originally given in \cite{GeissKellerOppermann-n-angulated-categories}, is the following.

\begin{defn}
\label{def:pre-and-n-angulated-category}

\cite[Def.\ 2.1]{GeissKellerOppermann-n-angulated-categories}
Let $(\CC,\sus,\CT)$ be a weak (pre-)$(n+2)$-angulated category. We call $(\CC,\sus,\CT)$ a \emph{(pre-)$(n+2)$-angulated category} if $\sus$ is an automorphism of $\CC$.
\end{defn}

Sometimes will use the adjective `strong' for a (pre-)$(n+2)$-angulated category (in the sense of Definition \ref{def:pre-and-n-angulated-category} above) in order to highlight that the $(n+2)$-suspension functor is an automorphism. 
As is normal practice, we will usually just talk of a (weak) (pre-)$(n+2)$-angulated category $\CC$ and assume the existence of $\CT$ and $\sus$ is understood.

\begin{example}
\label{exa:triangulated-category-as-special-case}

With $n=1$, an $(n+2)$-angulated category (in the sense of Definition \ref{def:pre-and-n-angulated-category}) is the same as a triangulated category in the sense of \cite[Def.\ 3.1]{HolmJorgensen-tri-cats-intro}. See \cite[Def.\ 3.1, Def.\ 3.2, pp.\ 246--248]{Neeman-Some-new-axioms-for-triangulated-categories}.

\end{example}

\begin{defn}\label{def:n-angulated-functor}
\cite[\S{}4]{BerghThaule-the-grothendieck-group-of-an-n-angulated-caegory}
Let $(\CC,\sus,\mathcal{T})$ and $(\CC',\sus',\mathcal{T}')$
be weak (pre-)$(n+2)$-angulated categories. An additive covariant functor $\SF\colon\CC\to\CC'$ is called an \emph{$(n+2)$-angulated} functor if there exists a natural isomorphism 
$\Theta = \{\Theta_{X}\}_{X\in\CC} \colon \SF \circ \sus \overset{\iso}{\Longrightarrow} \sus'\circ \SF$, such that 
if 
\[
\begin{tikzcd}[column sep=0.7cm]X^{0}\arrow{r}{d_{X}^{0}}&X^{1}\arrow{r}{d_{X}^{1}}&\cdots\arrow{r}{d_{X}^{n}}&X^{n+1}\arrow{r}{d_{X}^{n+1}}&\sus X^{0}\end{tikzcd}
\]
is an $(n+2)$-angle in $\CC$, then 
\[
\begin{tikzcd}[column sep=1.7cm] \SF X^{0}\arrow{r}{\SF d_{X}^{0}}&\SF X^{1}\arrow{r}{\SF d_{X}^{1}}&\cdots\arrow{r}{\SF d_{X}^{n}}&\SF X^{n+1}\arrow{r}{\Theta_{X^{0}}\circ\SF d_{X}^{n+1}}&\sus' \SF X^{0}\end{tikzcd}
\]
is an $(n+2)$-angle in $\CC'$.

\end{defn}

We remark that $(n+2)$-angulated functors as we refer to them here were called `exact' functors in \cite{Jasso-n-abelian-and-n-exact-categories}. The last piece of terminology we need for this setting is the following.

\begin{defn} 

An $(n+2)$-angulated functor $\SF\colon\CC\to\CC'$ is called an \emph{$(n+2)$-angle equivalence} if it is an equivalence of categories. 

\end{defn}


\medskip
\subsection{\texorpdfstring{$n$}{n}-exact and \texorpdfstring{$n$}{n}-abelian categories}
\label{sec:n-exact-n-abelian-categories}

In \S\ref{sec:n-exact-n-abelian-categories} we follow \cite{Jasso-n-abelian-and-n-exact-categories} in recalling some prerequisite concepts from higher homological algebra and the definitions of an $n$-exact category and an $n$-abelian category (see Definitions \ref{def:n-exact-category} and \ref{def:n-abelian-category}, respectively). 
At the end of \S\ref{sec:n-exact-n-abelian-categories} we define an $n$-exact functor (see Definition \ref{def:n-exact-functor}). 
We denote by $\Ab$ the category of all abelian groups.

\begin{defn}\cite[Def.\ 2.2]{Jasso-n-abelian-and-n-exact-categories}
Let $d_{X}^{n}\colon X^{n}\to X^{n+1}$ be a morphism in $\CC$. An \emph{$n$-kernel} of $d_{X}^{n}$ is a sequence
\[
\begin{tikzcd}
(d_{X}^{0},\ldots,d_{X}^{n-1})\colon\phantom{ll}X^{0}\arrow{r}{d_{X}^{0}}& X^{1} \arrow{r}{d_{X}^{1}}& \cdots \arrow{r}{d_{X}^{n-1}}& X^{n}
\end{tikzcd}
\]
of composable morphisms in $\CC$, such that for all $Z\in\CC$ the functor $\CC(Z,-)$ yields an exact sequence
\[
\begin{tikzcd}[column sep=1.3cm]
0 \arrow{r}& \CC(Z,X^{0}) \arrow{r}{d_{X}^{0}\circ -}& \CC(Z,X^{1})\arrow{r}{d_{X}^{1}\circ -}& \cdots \arrow{r}{d_{X}^{n-1}\circ -}& \CC(Z,X^{n})\arrow{r}{d_{X}^{n}\circ -}& \CC(Z,X^{n+1})
\end{tikzcd}
\]
in $\Ab$. 
An \emph{$n$-cokernel} is defined dually.
\end{defn}

Let $\com_{\CC}$ denote the \emph{category of (cochain) complexes} in $\CC$. 
We denote by $\com^{n}_{\CC}$ the full subcategory of $\com_{\CC}$ consisting of complexes which are concentrated in degrees $0,\ldots,n+1$. We denote a complex $X^{\raisebox{0.5pt}{\scalebox{0.6}{$\bullet$}}}\in\com^{n}_{\CC}$ by 
\[
\begin{tikzcd}
X^{0}\arrow{r}{d_{X}^{0}}& X^{1} \arrow{r}{d_{X}^{1}}& \cdots\arrow{r}{d_{X}^{n-1}}&X^{n}\arrow{r}{d_{X}^{n}}& X^{n+1}.
\end{tikzcd}
\]
\begin{defn}\cite[Def.\ 2.4]{Jasso-n-abelian-and-n-exact-categories} 
An \emph{$n$-exact sequence} in $\CC$ is a complex $X^{\raisebox{0.5pt}{\scalebox{0.6}{$\bullet$}}}\in\com^{n}_{\CC}$, 
such that $(d_{X}^{0},\ldots,d_{X}^{n-1})$ is an $n$-kernel of $d_{X}^{n}$ and $(d_{X}^{1},\ldots,d_{X}^{n})$ is an $n$-cokernel of $d_{X}^{0}$.
\end{defn}

\begin{defn}

\cite[Def.\ 2.9]{Jasso-n-abelian-and-n-exact-categories} 
Let $X^{\raisebox{0.5pt}{\scalebox{0.6}{$\bullet$}}},Y^{\raisebox{0.5pt}{\scalebox{0.6}{$\bullet$}}}\in\com^{n}_{\CC}$ be $n$-exact sequences. A \emph{morphism $f^{\raisebox{0.5pt}{\scalebox{0.6}{$\bullet$}}}\colon X^{\raisebox{0.5pt}{\scalebox{0.6}{$\bullet$}}}\to Y^{\raisebox{0.5pt}{\scalebox{0.6}{$\bullet$}}}$ of $n$-exact sequences} is just a morphism $f^{\raisebox{0.5pt}{\scalebox{0.6}{$\bullet$}}}\in\com^{n}_{\CC}(X^{\raisebox{0.5pt}{\scalebox{0.6}{$\bullet$}}},Y^{\raisebox{0.5pt}{\scalebox{0.6}{$\bullet$}}})$ of complexes.

\end{defn}

\begin{defn}
\label{def:mapping-cone-as-in-Jasso}

\cite[Def.\ 2.11]{Jasso-n-abelian-and-n-exact-categories}
Let $f^{\raisebox{0.5pt}{\scalebox{0.6}{$\bullet$}}}\colon X^{\raisebox{0.5pt}{\scalebox{0.6}{$\bullet$}}}\to Y^{\raisebox{0.5pt}{\scalebox{0.6}{$\bullet$}}}$ be a morphism of complexes in $\com^{n-1}_{\CC}$. The \emph{mapping cone} $C^{\raisebox{0.5pt}{\scalebox{0.6}{$\bullet$}}}\deff \cone(f)^{\raisebox{0.5pt}{\scalebox{0.6}{$\bullet$}}}$ is the complex
\[
\begin{tikzcd}
X^{0}\arrow{r}{d_{C}^{-1}}& X^{1}\oplus Y^{0} \arrow{r}{d_{C}^{0}}&X^{2}\oplus Y^{1}\arrow{r}{d_{C}^{1}}& \cdots\arrow{r}{d_{C}^{n-2}}&X^{n}\oplus Y^{n-1} \arrow{r}{d_{C}^{n-1}}& Y^{n}
\end{tikzcd}
\]
in $\com^{n}_{\CC}$, where $d_{C}^{-1}\deff \begin{psmallmatrix}-d_{X}^{0}\\f^{0} \end{psmallmatrix}$, 
$d_{C}^{n-1}\deff(\,f^{n}\;\;\, d_{Y}^{n-1})$, 
and  
\[
d_{C}^{i}\deff
\begin{pmatrix}-d_{X}^{i+1} & 0\\f^{i+1} & d_{Y}^{i} \end{pmatrix}
\]
for $i\in\{0,\ldots,n-2\}$.

\end{defn}

\begin{defn}

\cite[Def.\ 2.11]{Jasso-n-abelian-and-n-exact-categories}
Let $X^{\raisebox{0.5pt}{\scalebox{0.6}{$\bullet$}}}\in\com^{n-1}_{\CC}$ be a complex. Suppose $g^{n}\colon Z^{n}\to X^{n}$ is a morphism in $\CC$. An \emph{$n$-pullback of $(d_{X}^{0},\ldots,d_{X}^{n-1})$ along $g^{n}$} is a morphism of complexes
\[
\begin{tikzcd}
Z^{0}\arrow{r}{d^{0}_{Z}}\arrow{d}{g^{0}}& Z^{1} \arrow{r}{d^{1}_{Z}}\arrow{d}{g^{1}}& \cdots\arrow{r}{d^{n-2}_{Z}}&Z^{n-1}\arrow{r}{d^{n-1}_{Z}}\arrow{d}{g^{n-1}}& Z^{n}\arrow{d}{g^{n}}\\
X^{0}\arrow{r}[swap]{d^{0}_{X}}& X^{1} \arrow{r}[swap]{d^{1}_{X}}& \cdots\arrow{r}[swap]{d^{n-2}_{X}}&X^{n-1}\arrow{r}[swap]{d^{n-1}_{X}}& X^{n},
\end{tikzcd}
\]
such that in the mapping cone
\[
\begin{tikzcd}[column sep=0.6cm]
C^{\raisebox{0.5pt}{\scalebox{0.6}{$\bullet$}}}\deff \cone(g)^{\raisebox{0.5pt}{\scalebox{0.6}{$\bullet$}}}\colon&Z^{0}\arrow{r}{d_{C}^{-1}}& Z^{1}\oplus X^{0} \arrow{r}{d_{C}^{0}}&Z^{2}\oplus X^{1}\arrow{r}{d_{C}^{1}}& \cdots\arrow{r}{d_{C}^{n-2}}&Z^{n}\oplus X^{n-1} \arrow{r}{d_{C}^{n-1}}& X^{n}
\end{tikzcd}
\]
the sequence $(d_{C}^{-1},\ldots,d_{C}^{n-2})$ is an $n$-kernel of $d_{C}^{n-1}$. 
An \emph{$n$-pushout} is defined dually.

\end{defn}

\begin{defn}

\cite[\S4.1]{Jasso-n-abelian-and-n-exact-categories}
Fix a class $\CX$ of $n$-exact sequences in $\CC$. The members of $\CX$ are called \emph{$\CX$-admissible $n$-exact sequences}. If $X^{\raisebox{0.5pt}{\scalebox{0.6}{$\bullet$}}}$
is an $\CX$-admissible $n$-exact sequence, then $d_{X}^{0}\colon X^{0}\rightarrowtail X^{1}$ is called an \emph{$\CX$-admissible monomorphism} and 
$d_{X}^{n}\colon X^{n}\onto X^{n+1}$ is called an \emph{$\CX$-admissible epimorphism}.

\end{defn}

\begin{defn}

\cite[Def.\ 4.1]{Jasso-n-abelian-and-n-exact-categories}
Let $f^{\raisebox{0.5pt}{\scalebox{0.6}{$\bullet$}}}\colon X^{\raisebox{0.5pt}{\scalebox{0.6}{$\bullet$}}}\to Y^{\raisebox{0.5pt}{\scalebox{0.6}{$\bullet$}}}$ be a morphism of $n$-exact sequences in $\com^{n}_{\CC}$. If there exists $i\in\{0,1,\ldots,n+1\}$ such that $f^{i}$ and $f^{i+1}$ (where $f^{n+2}\deff f^{0}$) are isomorphisms, then we call $f^{\raisebox{0.5pt}{\scalebox{0.6}{$\bullet$}}}$ a \emph{weak isomorphism}, and we say $X^{\raisebox{0.5pt}{\scalebox{0.6}{$\bullet$}}}$ and $Y^{\raisebox{0.5pt}{\scalebox{0.6}{$\bullet$}}}$ are \emph{weakly isomorphic}.

\end{defn}

We are now able to recall the definitions of an $n$-exact category and an $n$-abelian category.

\begin{defn}
\label{def:n-exact-category}

\cite[Def.\ 4.2]{Jasso-n-abelian-and-n-exact-categories}
Let $n\geq 1$ be a positive integer and let $\CC$ be an additive category. Suppose $\CX$ is a class of $n$-exact sequences in $\CC$ that is closed under weak isomorphisms. The pair $(\CC,\CX)$ forms an \emph{$n$-exact category} if the following axioms are satisfied.
\begin{enumerate}[label=($n$-E\arabic*),wide=0pt, leftmargin=43pt, labelwidth=43pt, labelsep=0pt, align=left]
\setcounter{enumi}{-1}

    \item\label{nE0}
    The sequence $0\rightarrowtail 0 \to \cdots \to 0\onto 0$ is an $\CX$-admissible $n$-exact sequence.
    \item\label{nE1}
    The class of $\CX$-admissible monomorphisms is closed under composition.
    \item[($n$-E$1^{\op}$)]\label{nE1op}
    The class of $\CX$-admissible epimorphisms is closed under composition.
    \item\label{nE2}
    For each $\CX$-admissible $n$-exact sequence $X^{\raisebox{0.5pt}{\scalebox{0.6}{$\bullet$}}}$ and each morphism $f^{0}\colon X^{0}\to Y^{0}$, there exists an $n$-pushout
    \begin{center}
    	$
    \begin{tikzcd}
    X^{0}\arrow[tail]{r}{d^{0}_{X}}\arrow{d}{f^{0}}& X^{1} \arrow{r}{d^{1}_{X}}\arrow[dotted]{d}{f^{1}}& \cdots\arrow{r}{d_{X}^{n-2}}&X^{n-1}\arrow{r}{d^{n-1}_{X}}\arrow[dotted]{d}{f^{n-1}}&X^{n}\arrow[two heads]{r}{d^{n}_{X}}\arrow[dotted]{d}{f^{n}}& X^{n+1}\\
    Y^{0}\arrow[dotted]{r}[swap]{d^{0}_{Y}}& Y^{1} \arrow[dotted]{r}[swap]{d^{1}_{Y}}& \cdots\arrow[dotted]{r}[swap]{d_{Y}^{n-2}}&Y^{n-1}\arrow[dotted]{r}[swap]{d^{n-1}_{Y}}&Y^{n}
    \end{tikzcd}
    $
    \end{center}
    of $(d^{0}_{X},\ldots,d^{n-1}_{X})$ along $f^{0}$ such that $d^{0}_{Y}$ is an $\CX$-admissible monomorphism.
    \item[($n$-E$2^{\op}$)]\label{nE2op}
    Dual of \ref{nE2}.

\end{enumerate}

\end{defn}


\begin{defn}
\label{def:n-abelian-category}

\cite[Def.\ 3.1]{Jasso-n-abelian-and-n-exact-categories}
Let $n\geq 1$ be a positive integer. An additive category $\CC$ is called \emph{$n$-abelian} if the following axioms are satisfied.
\begin{enumerate}[label=($n$-A\arabic*),wide=0pt, leftmargin=43pt, labelwidth=43pt, labelsep=0pt, align=left]
\setcounter{enumi}{-1}

    \item\label{nA0}
    The category $\CC$ has split idempotents.
    \item\label{nA1}
    Every morphism in $\CC$ has an $n$-kernel and an $n$-cokernel.
    \item\label{nA2}
    For each monomorphism $d_{X}^{0}\colon X^{0} \into X^{1}$ in $\CC$ and every $n$-cokernel $(d_{X}^{1},\ldots,d_{X}^{n})$ of $d_{X}^{0}$, the sequence
    \begin{center}
    	$
    \begin{tikzcd}X^{0}\arrow{r}{d_{X}^{0}}& X^{1} \arrow{r}{d_{X}^{1}}& \cdots\arrow{r}{d_{X}^{n-1}}&X^{n}\arrow{r}{d_{X}^{n}}& X^{n+1}\end{tikzcd}
    	$
     \end{center}
    is $n$-exact.
    \item[($n$-A$2^{\op}$)]\label{nA2op}
    Dual of \ref{nA2}.

\end{enumerate}

\end{defn}

It was shown in \cite{Jasso-n-abelian-and-n-exact-categories} that every $n$-abelian category is $n$-exact; see \cite[Thm.\ 4.4]{Jasso-n-abelian-and-n-exact-categories}. 

For an additive covariant functor $\SF\colon \CA\to \CB$ between additive categories, we denote by $\SF_{\com}\colon \com_{\CA}\to \com_{\CB}$ the induced additive covariant functor between the categories of complexes. 
More precisely, for complexes $X^{\raisebox{0.5pt}{\scalebox{0.6}{$\bullet$}}},Y^{\raisebox{0.5pt}{\scalebox{0.6}{$\bullet$}}}\in\com_{\CA}$ and any morphism $f^{\raisebox{0.5pt}{\scalebox{0.6}{$\bullet$}}}\in\com_{\CA}(X^{\raisebox{0.5pt}{\scalebox{0.6}{$\bullet$}}},Y^{\raisebox{0.5pt}{\scalebox{0.6}{$\bullet$}}})$, the complex 
$\SF_{\com}(X^{\raisebox{0.5pt}{\scalebox{0.6}{$\bullet$}}})\in\com_{\CB}$ is 
\[
\begin{tikzcd}
\cdots \arrow{r} & \SF X^{-1} \arrow{r}{\SF d_{X}^{-1}} & \SF X^{0} \arrow{r}{\SF d_{X}^{0}} &\SF X^{1} \arrow{r}{\SF d_{X}^{1}} &\cdots
\end{tikzcd}
\] 
and the morphism $\SF_{\com}(f^{\raisebox{0.5pt}{\scalebox{0.6}{$\bullet$}}})\in\com_{\CB}(\SF_{\com}X^{\raisebox{0.5pt}{\scalebox{0.6}{$\bullet$}}},\SF_{\com}Y^{\raisebox{0.5pt}{\scalebox{0.6}{$\bullet$}}})$ is given by the commutative diagram 
\[
\begin{tikzcd}[column sep=1.3cm]
\cdots \arrow{r} & \SF X^{-1} \arrow{r}{\SF d_{X}^{-1}} \arrow{d}{\SF f^{-1}}& \SF X^{0} \arrow{r}{\SF d_{X}^{0}} \arrow{d}{\SF f^{0}}&\SF X^{1} \arrow{r}{\SF d_{X}^{1}} \arrow{d}{\SF f^{1}}&\cdots\\
 \cdots \arrow{r} & \SF Y^{-1} \arrow{r}{\SF d_{Y}^{-1}} & \SF Y^{0} \arrow{r}{\SF d_{Y}^{0}} &\SF Y^{1} \arrow{r}{\SF d_{Y}^{1}} &\cdots .
\end{tikzcd}
\]

Now we introduce the higher version of an exact functor between exact categories; see \cite[Def.\ 5.1]{Buhler-exact-categories}. 

\begin{defn}
\label{def:n-exact-functor}

Let $(\CC, \CX)$ and $(\CC', \CX')$ be $n$-exact categories. An additive covariant functor $\SF\colon \CC\to\CC'$ is called an \emph{$n$-exact} functor if it sends $\CX$-admissible $n$-exact sequences to $\CX'$-admissible $n$-exact sequences, i.e. $\SF_{\com}(\CX)\subseteq \CX'$. If $\SF$ is $n$-exact and an equivalence, then we call $\SF$ an \emph{$n$-exact equivalence}.

\end{defn}

\begin{rem}

We note here that the terminology `$n$-exact functor' has appeared in \cite[\S{}4.1]{Luo-homological-algebra-in-n-abelian-categories}. 
However, in \cite{Luo-homological-algebra-in-n-abelian-categories} an `$n$-exact functor' is an additive covariant functor from an $n$-abelian category to an \emph{abelian} category that maps $n$-exact sequences to exact sequences. On the other hand, what we call an $n$-exact functor is a direct generalisation of an exact functor (between exact categories).

\end{rem}


\medskip
\subsection{\texorpdfstring{$n$}{n}-exangulated categories}
\label{sec:n-exangulated-categories}

In \S\ref{sec:n-exangulated-categories} we follow \cite[\S2]{HerschendLiuNakaoka-n-exangulated-categories-I-definitions-and-fundamental-properties} in order to recall the definition of an $n$-exangulated category (see Definition \ref{def:n-exangulated-category}). 
We also recall how one can view $(n+2)$-angulated and $n$-exact categories as $n$-exangulated categories (see Examples \ref{exa:n+2-angulated-category-is-n-exangulated} and \ref{exa:n-exact-category-is-n-exangulated}, respectively). 
Lastly, we define an $n$-exangulated functor (see Definition \ref{def:n-exangulated-functor}), and we relate this notion to $(n+2)$-angulated functors and $n$-exact functors (see Theorems \ref{thm:n+2-angulated-functor-iff-n-exangulated} and \ref{thm:n-exact-functor-iff-n-exangulated}, respectively).

\begin{setup}
\label{ass1}

We assume $\CC$ comes equipped with a biadditive functor $\BE\colon\CC^{\op}\times\CC\to \Ab$ (which is the same as an `additive bifunctor' in the sense of \cite[\S{}1.2, p.\ 649]{DraxlerReitenSmaloSolberg-exact-categories-and-vector-space-categories}). 
Thus, for fixed objects $A,C$ in $\CC$, the functors $\BE(C,-)\colon \CC \to \Ab$ and $\BE(-,A)\colon \CC^{\op}\to\Ab$ are additive. 
Furthermore, each morphism $f\colon X\to Y$ in $\CC$ gives rise to abelian group homomorphisms  $\BE(C,f)\colon \BE(C,X)\to \BE(C,Y)$ and $\BE(f^{\op},A)\colon \BE(Y,A)\to\BE(X,A)$. 
To simplify notation, we will omit the superscript ``$^{\op}$'', writing $\BE(f,A)$ instead of $\BE(f^{\op},A)$. 

\end{setup}

\begin{defn}

\cite[Def.\ 2.1, Rem. 2.2, Def.\ 2.3]{HerschendLiuNakaoka-n-exangulated-categories-I-definitions-and-fundamental-properties} 
By an \emph{$\BE$-extension} (or simply an \emph{extension}) we mean an element $\delta$ of $\BE(C,A)$ for some objects $A,C$ in $\CC$. 

For morphisms $x\colon A\to X$, $z\colon Z\to C$ in $\CC$ and for an extension $\delta\in\BE(C,A)$, we obtain new extensions as follows: 
\[
x_{\BE}\delta:=\BE(C,x)(\delta)\in\BE(C,X)
\hspace{1cm}
\text{and} 
\hspace{1cm}
z^{\BE}\delta:=\BE(z,A)(\delta)\in\BE(Z,A).
\] 

Let $\delta\in \BE(C,A)$ and $\eps\in\BE(D,B)$ be extensions. A \emph{morphism of $\BE$-extensions} $\delta\to\eps$ is a pair $(a,c)$ of morphisms $a\colon A\to B$ and $c\colon C\to D$ in $\CC$, such that $ a_{\BE}\delta  = c^{\BE}\eps$.

\end{defn}

By the Yoneda Lemma, each $\BE$-extension gives rise to two natural transformations as follows.

\begin{defn}

\cite[Def.\ 2.11]{HerschendLiuNakaoka-n-exangulated-categories-I-definitions-and-fundamental-properties} 
For any $A,C$ in $\CC$ and any $\BE$-extension $\delta\in\BE(C,A)$, there is an induced natural transformation 
$\delta^{\sharp}_{\BE}\colon \CC(A,-) \Rightarrow \BE(C,-)$, 
which is given by 
$(\delta^{\sharp}_{\BE})_{B}(a)=a_{\BE}\delta$ 
for each object $B\in\CC$ and each morphism $a\colon A\to B$. 
Dually, there is also a natural transformation 
$\delta_{\sharp}^{\BE}\colon \CC(-,C)\Rightarrow \BE(-,A)$, 
which is given by 
$(\delta_{\sharp}^{\BE})_{D}(d)=d^{\BE}\delta$ 
for each object $D\in\CC$ and each morphism $d\colon D\to C$.

\end{defn}

We will be particularly interested in pairs $\lan X^{\raisebox{0.5pt}{\scalebox{0.6}{$\bullet$}}}, \delta\ran$ of complexes and extensions for which the natural transformations ${}\delta^{\sharp}_{\BE}$ and $\delta_{\sharp}^{\BE}$ fit into exact sequences induced by $\Hom$-functors.

\begin{defn}
\label{def:n-exangles}

\cite[Def.\ 2.13]{HerschendLiuNakaoka-n-exangulated-categories-I-definitions-and-fundamental-properties} 
A pair $\langle X^{\raisebox{0.5pt}{\scalebox{0.6}{$\bullet$}}},\delta\rangle$, where $X^{\raisebox{0.5pt}{\scalebox{0.6}{$\bullet$}}}\in\com^{n}_{\CC}$ and $\delta\in\BE(X^{n+1},X^{0})$, is called an $n$-\emph{exangle} if, for every object $Z$ in $\CC$, the sequences 
\[
\begin{tikzcd}[column sep=1cm]\CC(Z,X^{0})\arrow{r}{d_{X}^{0}\circ -}&\CC(Z,X^{1})\arrow{r}{d_{X}^{1}\circ -}&\cdots\arrow{r}{d_{X}^{n}\circ -}&\CC(Z,X^{n+1})\arrow{r}{(\delta_{\sharp}^{\BE})_{Z}}& \BE(Z,X^{0})\end{tikzcd}
\] 
and 
\[
\begin{tikzcd}[column sep=1cm]\CC(X^{n+1},Z)\arrow{r}{-\circ d_{X}^{n}}&\CC(X^{n},Z)\arrow{r}{-\circ d_{X}^{n-1}}&\cdots\arrow{r}{-\circ d_{X}^{0}}&\CC(X^{0},Z)\arrow{r}{(\delta^{\sharp}_{\BE})_{Z}}& \BE(X^{n+1},Z)\end{tikzcd}
\] 
in $\Ab$ are exact. 

\end{defn}

In \cite{HerschendLiuNakaoka-n-exangulated-categories-I-definitions-and-fundamental-properties} and \cite{NakaokaPalu-extriangulated-categories-hovey-twin-cotorsion-pairs-and-model-structures} the notation $x_{*}\delta$, respectively $z^{*}\delta$, is used to label the extensions $x_{\BE}\delta$, respectively $z^{\BE}\delta$. The reason for our change in notation become apparent later; see Remark \ref{rem:E-notation-changes}.

\begin{defn}

\cite[Def.\ 2.17]{HerschendLiuNakaoka-n-exangulated-categories-I-definitions-and-fundamental-properties} 
Let $A,C$ be objects in $\CC$. 
We define the subcategory $\com^{n}_{(A,C)}$ of $\com^{n}_{\CC}$ in the following way. 
The objects of $\com^{n}_{(A,C)}$ consist of those complexes $X^{\raisebox{0.5pt}{\scalebox{0.6}{$\bullet$}}}\in\com_{\CC}^{n}$ for which $X^{0}=A$ and $X^{n+1}=C$. 
For $X^{\raisebox{0.5pt}{\scalebox{0.6}{$\bullet$}}},Y^{\raisebox{0.5pt}{\scalebox{0.6}{$\bullet$}}}\in\com^{n}_{(A,C)}$, we set
\[
\com^{n}_{(A,C)}(X^{\scalebox{0.6}{$\bullet$}},Y^{\scalebox{0.6}{$\bullet$}})\deff
\{
f^{\scalebox{0.6}{$\bullet$}}\in \com^{n}_{\CC}(X^{\scalebox{0.6}{$\bullet$}},Y^{\scalebox{0.6}{$\bullet$}})\mid f^{0}=1_{A}\text{ and }f^{n+1}=1_{C}
\}.
\] 

\end{defn}

For morphisms $f^{\raisebox{0.5pt}{\scalebox{0.6}{$\bullet$}}},g^{\raisebox{0.5pt}{\scalebox{0.6}{$\bullet$}}}\colon X^{\raisebox{0.5pt}{\scalebox{0.6}{$\bullet$}}}\to Y^{\raisebox{0.5pt}{\scalebox{0.6}{$\bullet$}}}$ in $\com_{\CC}$, we write $f^{\raisebox{0.5pt}{\scalebox{0.6}{$\bullet$}}}\sim g^{\raisebox{0.5pt}{\scalebox{0.6}{$\bullet$}}}$ if $f^{\raisebox{0.5pt}{\scalebox{0.6}{$\bullet$}}}$ and $g^{\raisebox{0.5pt}{\scalebox{0.6}{$\bullet$}}}$ are \emph{homotopic} in the usual sense. 
Recall that $\sim$ gives an equivalence relation on $\com_{\CC}(X^{\raisebox{0.5pt}{\scalebox{0.6}{$\bullet$}}},Y^{\raisebox{0.5pt}{\scalebox{0.6}{$\bullet$}}})$ for all complexes  $X^{\raisebox{0.5pt}{\scalebox{0.6}{$\bullet$}}},Y^{\raisebox{0.5pt}{\scalebox{0.6}{$\bullet$}}}$. 
Now fix $A,C\in\CC$ and let $X^{\raisebox{0.5pt}{\scalebox{0.6}{$\bullet$}}},Y^{\raisebox{0.5pt}{\scalebox{0.6}{$\bullet$}}}\in\com^{n}_{(A,C)}$ be arbitrary. 
Note that the set $\com^{n}_{(A,C)}(X^{\raisebox{0.5pt}{\scalebox{0.6}{$\bullet$}}},Y^{\raisebox{0.5pt}{\scalebox{0.6}{$\bullet$}}})$ is no longer an abelian group (see \cite[p.\ 540]{HerschendLiuNakaoka-n-exangulated-categories-I-definitions-and-fundamental-properties}). 
However, the homotopy relation $\sim$ on $\com^{n}_{\CC}(X^{\raisebox{0.5pt}{\scalebox{0.6}{$\bullet$}}},Y^{\raisebox{0.5pt}{\scalebox{0.6}{$\bullet$}}})$ restricts to an equivalence relation on $\com^{n}_{(A,C)}(X^{\raisebox{0.5pt}{\scalebox{0.6}{$\bullet$}}},Y^{\raisebox{0.5pt}{\scalebox{0.6}{$\bullet$}}})$, which we again denote by $\sim$.

\begin{defn}

\cite[\S{}2.3]{HerschendLiuNakaoka-n-exangulated-categories-I-definitions-and-fundamental-properties} 
By $\kom^{n}_{(A,C)}$ we denote the category whose objects are those of $\com^{n}_{(A,C)}$ and, for $X^{\raisebox{0.5pt}{\scalebox{0.6}{$\bullet$}}},Y^{\raisebox{0.5pt}{\scalebox{0.6}{$\bullet$}}}\in\kom^{n}_{(A,C)}$, we put  
\[
\kom^{n}_{(A,C)}(X^{\scalebox{0.6}{$\bullet$}},Y^{\scalebox{0.6}{$\bullet$}})\deff \com^{n}_{(A,C)}(X^{\scalebox{0.6}{$\bullet$}},Y^{\scalebox{0.6}{$\bullet$}})/\sim.
\] 

\end{defn}

Given a morphism $f^{\raisebox{0.5pt}{\scalebox{0.6}{$\bullet$}}}\in\com^{n}_{(A,C)}(X^{\raisebox{0.5pt}{\scalebox{0.6}{$\bullet$}}},Y^{\raisebox{0.5pt}{\scalebox{0.6}{$\bullet$}}})$, we will call $f^{\raisebox{0.5pt}{\scalebox{0.6}{$\bullet$}}}$ a \emph{homotopy equivalence} if it induces an isomorphism in $\kom^{n}_{(A,C)}(X^{\raisebox{0.5pt}{\scalebox{0.6}{$\bullet$}}},Y^{\raisebox{0.5pt}{\scalebox{0.6}{$\bullet$}}})$. In this case, we will say the complexes $X^{\raisebox{0.5pt}{\scalebox{0.6}{$\bullet$}}}$ and $Y^{\raisebox{0.5pt}{\scalebox{0.6}{$\bullet$}}}$ are \emph{homotopy equivalent}. We denote by $[X^{\raisebox{0.5pt}{\scalebox{0.6}{$\bullet$}}}]$ the homotopy equivalence class of $X^{\raisebox{0.5pt}{\scalebox{0.6}{$\bullet$}}}$ in $\com^{n}_{(A,C)}$. 
(Note that this may not coincide with the usual homotopy equivalence class of $X^{\raisebox{0.5pt}{\scalebox{0.6}{$\bullet$}}}$ in $\com_{\CC}^{n}$; see \cite[Rem. 2.18]{HerschendLiuNakaoka-n-exangulated-categories-I-definitions-and-fundamental-properties}.)

\begin{defn}
\label{def:exact-realisation}

\cite[Def.\ 2.22]{HerschendLiuNakaoka-n-exangulated-categories-I-definitions-and-fundamental-properties} 
Let $\fs$ be a correspondence that, for each $A,C\in\CC$, assigns to each extension $\delta\in\BE(C,A)$ a homotopy equivalence class $\fs(\delta)=[X^{\raisebox{0.5pt}{\scalebox{0.6}{$\bullet$}}}]$ where $X^{\raisebox{0.5pt}{\scalebox{0.6}{$\bullet$}}}\in\com^{n}_{(A,C)}$. 
We say that $\fs$ is a \emph{realisation of $\BE$} if 
\begin{enumerate}[label=(R\arabic*)]
\setcounter{enumi}{-1}

    \item\label{R0}
    for all $\delta\in\BE(C,A)$ and $\eps\in\BE(D,B)$ with $\fs(\delta)=[X^{\raisebox{0.5pt}{\scalebox{0.6}{$\bullet$}}}]$ and $\fs(\eps)=[Y^{\raisebox{0.5pt}{\scalebox{0.6}{$\bullet$}}}]$, and for all morphisms $(a,c)\colon \delta\to\eps$ of $\BE$-extensions, there exists a morphism $f^{\raisebox{0.5pt}{\scalebox{0.6}{$\bullet$}}}=(f^{0},\ldots,f^{n+1})\in\com^{n}_{\CC}(X^{\raisebox{0.5pt}{\scalebox{0.6}{$\bullet$}}},Y^{\raisebox{0.5pt}{\scalebox{0.6}{$\bullet$}}})$ with $f^{0}=a$ and $f^{n+1}=c$.

\end{enumerate} 
In this case, we say 
$X^{\raisebox{0.5pt}{\scalebox{0.6}{$\bullet$}}}$ \emph{realises} $\delta$ 
(or $\delta$ is \emph{realised by} $X^{\raisebox{0.5pt}{\scalebox{0.6}{$\bullet$}}}$) 
and that $f^{\raisebox{0.5pt}{\scalebox{0.6}{$\bullet$}}}$ \emph{realises} $(a,c)$
(or $(a,c)$ is \emph{realised by} $f^{\raisebox{0.5pt}{\scalebox{0.6}{$\bullet$}}}$). 
For objects $A,C\in\CC$, we write $_{A}0_{C}$ for the identity element in the abelian group $\BE(C,A)$. 

A realisation $\fs$ of $\BE$ is said to be \emph{exact} if 
\begin{enumerate}[label=(R\arabic*)]

    \item\label{R1}
    the pair $\langle X^{\raisebox{0.5pt}{\scalebox{0.6}{$\bullet$}}},\delta\rangle$ is an $n$-exangle whenever $\fs(\delta)=[X^{\raisebox{0.5pt}{\scalebox{0.6}{$\bullet$}}}]$; and 
    \item\label{R2}
    for any object $A$ in $\CC$, we have that 
    $\fs(_{A}0_{0})=[
    \begin{tikzcd}[column sep=0.7cm] A\arrow{r}{1_{A}}&A\arrow{r}&0\arrow{r}&\cdots\arrow{r}&0\end{tikzcd}
    ]$ 
    and 
    $\fs(_{0}0_{A})=[
    \begin{tikzcd}[column sep=0.7cm] 0\arrow{r}&\cdots\arrow{r}&0\arrow{r}&A\arrow{r}{1_{A}}&A\end{tikzcd}
    ].$
    
\end{enumerate}

\end{defn}

Let us recall some useful terminology from \cite{HerschendLiuNakaoka-n-exangulated-categories-I-definitions-and-fundamental-properties}.

\begin{defn}

\cite[Def.\ 2.23]{HerschendLiuNakaoka-n-exangulated-categories-I-definitions-and-fundamental-properties} 
Let $\fs$ be an exact realisation of $\BE$. 
Suppose $\fs(\delta)=[X^{\raisebox{0.5pt}{\scalebox{0.6}{$\bullet$}}}]$ for some extension $\delta\in\BE(C,A)$ and $X^{\raisebox{0.5pt}{\scalebox{0.6}{$\bullet$}}}\in\com^{n}_{(A,C)}$. 
Then we say $\lan X^{\raisebox{0.5pt}{\scalebox{0.6}{$\bullet$}}},\delta\ran$ is an \emph{$\fs$-distinguished $n$-exangle} and $X^{\raisebox{0.5pt}{\scalebox{0.6}{$\bullet$}}}$ is an \emph{$\fs$-conflation}. 
A morphism $f$ in $\CC$ is called an \emph{$\fs$-inflation} (respectively, \emph{$\fs$-deflation}) if there exists an $\fs$-conflation $X^{\raisebox{0.5pt}{\scalebox{0.6}{$\bullet$}}}$ with $d_{X}^{0}=f$ (respectively, $d_{X}^{n}=f$).

\end{defn}

We are now in a position to recall the definition of an $n$-exangulated category.

\begin{defn}
\label{def:n-exangulated-category}

\cite[Def.\ 2.32]{HerschendLiuNakaoka-n-exangulated-categories-I-definitions-and-fundamental-properties} 
Recall that $n\geq 1$ is a positive integer, $\CC$ is an additive category and $\BE\colon \CC^{\op}\times\CC\to \Ab$ is a biadditive functor. Let $\fs$ be an exact realisation of $\BE$. Then $(\CC,\BE,\fs)$ is called an \emph{$n$-exangulated category} if the following axioms are satisfied. 
\begin{enumerate}[label=($n$-EA\arabic*),wide=0pt, leftmargin=50pt, labelwidth=50pt, labelsep=0pt, align=left]

    \item\label{nEA1}
    The class of $\fs$-inflations is closed under composition. Dually, the class of $\fs$-deflations is closed under composition.
    \item\label{nEA2}
    For each $\delta\in\BE(D,A)$ and each $c\in\CC(C,D)$, if $\fs(c^{\BE}\delta)=[X^{\raisebox{0.5pt}{\scalebox{0.6}{$\bullet$}}}]$ and $\fs(\delta)=[Y^{\raisebox{0.5pt}{\scalebox{0.6}{$\bullet$}}}]$, then there exists a morphism $f^{\raisebox{0.5pt}{\scalebox{0.6}{$\bullet$}}}=(1_{A},f^{1},\ldots,f^{n},c)\colon X^{\raisebox{0.5pt}{\scalebox{0.6}{$\bullet$}}}\to Y^{\raisebox{0.5pt}{\scalebox{0.6}{$\bullet$}}}$ realising $(1_{A},c)\colon c^{\BE}\delta\to\delta$ such that $\fs((d_{X}^{0})_{\BE}\delta)=[ \cone(\wh{f})^{\raisebox{0.5pt}{\scalebox{0.6}{$\bullet$}}}]$, where $\wh{f}^{\raisebox{0.5pt}{\scalebox{0.6}{$\bullet$}}}=(f^{1},\ldots,f^{n},c)$. Such an $f^{\raisebox{0.5pt}{\scalebox{0.6}{$\bullet$}}}$ is called a \emph{good lift} of $(1_{A},c)$.
    \item[($n$-EA$2^{\op}$)]\label{nEA2op}
    Dual of \ref{nEA2}.
    
\end{enumerate}

\end{defn}

\begin{example}

A category is $1$-exangulated if and only if it is \emph{extriangulated} (in the sense of \cite{NakaokaPalu-extriangulated-categories-hovey-twin-cotorsion-pairs-and-model-structures}); see \cite[Prop.\ 4.3]{HerschendLiuNakaoka-n-exangulated-categories-I-definitions-and-fundamental-properties}.

\end{example}

Let us recall how $(n+2)$-angulated and $n$-exact categories provide examples of $n$-exangulated categories.

\begin{example}
\label{exa:n+2-angulated-category-is-n-exangulated}

Suppose $(\CC,\sus,\CT)$ is an $(n+2)$-angulated category. 
From this information, we obtain a functor $\BE_{\sus}\colon \CC^{\op}\times\CC\to\Ab$ by setting $\BE_{\sus}(X^{n+1},X^{0})\deff \CC(X^{n+1},\sus X^{0})$ for any $X^{0},X^{n+1}$ in $\CC$, and defining the map $\BE_{\sus}(f,g)\colon \BE_{\sus}(X^{n+1},X^{0}) \to \BE_{\sus}(Y^{n+1},Y^{0})$ by $\BE_{\sus}(f,g)(\delta)\deff (\sus g)\circ \delta\circ f$ for any $f\colon Y^{n+1}\to X^{n+1}$ and $g\colon X^{0}\to Y^{0}$. 
An exact realisation $\fs$ of $\BE_{\sus}$ is given as follows: 
for an element $\delta\in\BE_{\sus}(X^{n+1},X^{0})$, complete it to an $(n+2)$-angle 
$X^{0} \to \cdots \to X^{n+1} \overset{\delta}{\to} \sus X^{0}$ and set $\fs(\delta) = [X^{0}\to\cdots\to X^{n+1}]$. 
Then $(\CC,\BE_{\sus},\fs)$ is an $n$-exangulated category. 
See \cite[\S{}4.2]{HerschendLiuNakaoka-n-exangulated-categories-I-definitions-and-fundamental-properties} for more details.

\end{example}

\begin{example}
\label{exa:n-exact-category-is-n-exangulated}

Suppose $(\CC,\CX)$ is an $n$-exact category. 
For any $X^{0},X^{n+1}$ in $\CC$, denote by $\Lambda^{n}_{(X^{0},X^{n+1})}$ the class of all homotopy equivalence classes of $n$-exact sequences in $\com^{n}_{(X^{0},X^{n+1})}$. 
Then define $\BE_{\CX}(X^{n+1},X^{0})$ to be the subclass of $\Lambda^{n}_{(X^{0},X^{n+1})}$ consisting of those classes $[X^{\raisebox{0.5pt}{\scalebox{0.6}{$\bullet$}}}]$ such that $X^{\raisebox{0.5pt}{\scalebox{0.6}{$\bullet$}}}\in\CX$. 
We assume that $\BE_{\CX}(X^{n+1},X^{0})$ is a set for all $X^{0},X^{n+1}\in\CC$. 
For any $\CX$-admissible $n$-exact sequence $X^{\raisebox{0.5pt}{\scalebox{0.6}{$\bullet$}}}$ and any morphism $f\colon X^{0}\to Y^{0}$, we can form an $n$-pushout of $(d_{X}^{0},\ldots,d_{X}^{n-1})$ along $f$ to obtain a morphism 
\[
\begin{tikzcd}
    X^{0}\arrow[tail]{r}{d^{0}_{X}}\arrow{d}{f}& X^{1} \arrow{r}{d^{1}_{X}}\arrow[dotted]{d}{f^{1}}& \cdots\arrow{r}{d^{n-1}_{X}}&X^{n}\arrow[two heads]{r}{d^{n}_{X}}\arrow[dotted]{d}{f^{n}}& X^{n+1}\arrow[equals]{d}\\
    Y^{0}\arrow[tail,dotted]{r}[swap]{d^{0}_{Y}}& Y^{1} \arrow[dotted]{r}[swap]{d^{1}_{Y}}& \cdots\arrow[dotted]{r}[swap]{d^{n-1}_{Y}}&Y^{n}\arrow{r}[swap]{d_{Y}^{n}} & X^{n+1}
    \end{tikzcd}
    \]
of $\CX$-admissible $n$-exact sequences by \cite[Prop.\ 4.8]{Jasso-n-abelian-and-n-exact-categories}. 
Thus, set $\BE_{\CX}(X^{n+1},f)([X^{\raisebox{0.5pt}{\scalebox{0.6}{$\bullet$}}}])=[Y^{\raisebox{0.5pt}{\scalebox{0.6}{$\bullet$}}}]$. We can define $\BE_{\CX}(g,X^{0})$ in a dual manner. 
Then $\BE_{\CX}\colon \CC^{\op}\times \CC \to\Ab$ is a well-defined biadditive functor. 
Lastly, the assignment $\fs(\delta)=[X^{\raisebox{0.5pt}{\scalebox{0.6}{$\bullet$}}}]$ for each $\delta=[X^{\raisebox{0.5pt}{\scalebox{0.6}{$\bullet$}}}]\in\BE_{\CX}(X^{n+1},X^{0})$ gives an exact realisation of $\BE_{\CX}$, and  $(\CC,\BE_{\CX},\fs)$ is an $n$-exangulated category. See \cite[\S{}4.3]{HerschendLiuNakaoka-n-exangulated-categories-I-definitions-and-fundamental-properties} for more details.

\end{example}

For any functor $\SF\colon \CA\to\CB$, we denote by $\SF^{\op}$ the \emph{opposite} functor $\CA^{\op}\to \CB^{\op}$ given by $\SF^{\op}(A)=\SF A$ and $\SF(f^{\op})=(\SF f)^{\op}$ for any objects $A,B\in\CA$ and any morphism $f\in\CA(A,B)$; see \cite[\S{}IV.5]{MacLaneMoerdijk-sheaves-in-geometry-and-logic}. 

We now introduce the notion of an $n$-exangulated functor, which appears to be new. 

\begin{defn}
\label{def:n-exangulated-functor}

Let $(\CC,\BE,\fs)$ and $(\CC',\BE',\fs')$ be $n$-exangulated categories. An additive covariant functor $\SF\colon \CC\to\CC'$ is called an \emph{$n$-exangulated} functor if 
there exists a natural transformation 
\[
\Gamma=\{\Gamma_{(C,A)}\}_{(C,A)\in\CC^{\op}\times\CC}\colon \BE(-,-) \Longrightarrow \BE'(\SF^{\op}-,\SF-)
\] 
of functors $\CC^{\op}\times\CC \to \Ab$, such that 
$\fs(\delta)=[X^{\raisebox{0.5pt}{\scalebox{0.6}{$\bullet$}}}]$ implies 
$\fs'(\Gamma_{(X^{n+1},X^{0})}(\delta))=[\SF_{\com} X^{\raisebox{0.5pt}{\scalebox{0.6}{$\bullet$}}}]$. 
To simplify notation, we write $\BE'(\SF-,\SF-)$ in place of $\BE'(\SF^{\op}-,\SF-)$. 
An $n$-exangulated functor $\SF\colon\CC\to\CC'$ is called an \emph{$n$-exangle equivalence} if it is an equivalence of categories.

In particular, we call a $1$-exangulated functor an \emph{extriangulated} functor.

\end{defn}

The following two results show that this definition specialises to recover the definitions of an $(n+2)$-angulated functor and an $n$-exact functor.

\begin{thm}
\label{thm:n+2-angulated-functor-iff-n-exangulated}

Let $(\CC,\sus,\CT)$ and $(\CC',\sus',\CT')$ be $(n+2)$-angulated categories. 
Following the notation of \emph{Example \ref{exa:n+2-angulated-category-is-n-exangulated}}, set $\BE\deff \BE_{\sus}$ and $\BE'\deff \BE_{\sus'}$, and consider the $n$-exangulated categories $(\CC,\BE,\fs)$ and $(\CC',\BE',\fs')$. 
Suppose $\SF\colon \CC\to\CC'$ is an additive covariant functor. Then $\SF$ is an $(n+2)$-angulated functor if and only if $\SF$ is an $n$-exangulated functor.

\end{thm}

\begin{proof}

$(\Rightarrow)$ Suppose first that $\SF$ is an $(n+2)$-angulated functor. Then there exists a natural isomorphism $\Theta\colon \SF\sus \overset{\iso}{\Longrightarrow} \sus'\SF$, such that for any $(n+2)$-angle 
$\begin{tikzcd}[column sep=0.7cm]X^{0}\arrow{r}{d_{X}^{0}}&X^{1}\arrow{r}{d_{X}^{1}}&\cdots\arrow{r}{d_{X}^{n}}&X^{n+1}\arrow{r}{d_{X}^{n+1}}&\sus X^{0}\end{tikzcd}$ in $\CC$, 
\[
\begin{tikzcd}[column sep=1.9cm] \SF X^{0}\arrow{r}{\SF d_{X}^{0}}&\SF X^{1}\arrow{r}{\SF d_{X}^{1}}&\cdots\arrow{r}{\SF d_{X}^{n}}&\SF X^{n+1}\arrow{r}{\Theta_{X^{0}}\circ\SF d_{X}^{n+1}}&\sus' \SF X^{0}\end{tikzcd}
\]
is an $(n+2)$-angle in $\CC'$.

For each $X^{0},X^{n+1}$ in $\CC$, let $\Gamma_{(X^{n+1},X^{0})}$ be the composition:
\[
\begin{tikzcd}[column sep = 1.6cm]
\CC(X^{n+1},\sus X^{0})\arrow{r}{\SF_{X^{n+1},\sus X^{0}}}&\CC'(\SF X^{n+1},\SF \sus X^{0})\arrow{r}{\Theta_{X^{0}}\circ -} & \CC'(\SF X^{n+1},\sus'\SF X^{0}).\end{tikzcd}
\]
This gives a homomorphism $\Gamma_{(X^{n+1},X^{0})}\colon \BE(X^{n+1},X^{0}) \to \BE'(\SF X^{n+1},\SF X^{0})$. 
We claim that the collection 
$\Gamma=\{\Gamma_{(X^{n+1},X^{0})}\}_{(X^{n+1},X^{0})\in\CC^{\op}\times\CC}$ gives a natural transformation $\BE(-,-)\Longrightarrow\BE'(\SF-,\SF-)$.

To this end, suppose $f\colon Y^{n+1}\to X^{n+1}$ and $g\colon X^{0}\to Y^{0}$ are arbitrary morphisms in $\CC$. Then we must show that the square
\begin{equation}\label{eqn:n+2-angulated-functor-implies-n-exangulated-functor-natural-transformation}
\begin{tikzcd}[column sep=2cm]
\BE(X^{n+1},X^{0}) \arrow{r}{\Gamma_{(X^{n+1},X^{0})}}\arrow{d}{\BE(f,g)}& \BE'(\SF X^{n+1},\SF X^{0}) \arrow{d}{\BE'(\SF f, \SF g)}\\
\BE(Y^{n+1},Y^{0}) \arrow{r}{\Gamma_{(Y^{n+1},Y^{0})}}& \BE'(\SF Y^{n+1},\SF Y^{0} )
\end{tikzcd}\end{equation}
commutes. 
Note that for any $\delta\in \BE(X^{n+1},X^{0})$ we have
\[
\BE'(\SF f,\SF g)(\Gamma_{(X^{n+1},X^{0})}(\delta)) 
= \BE'(\SF f,\SF g)(\Theta_{X^{0}}\SF \delta) 
= (\sus'\SF g)\Theta_{X^{0}}(\SF \delta)(\SF f).
\] 
On the other hand, 
\[
\Gamma_{(Y^{n+1},Y^{0})}(\BE(f,g)(\delta)) 
= \Gamma_{(Y^{n+1},Y^{0})}((\sus g) \delta f) 
= \Theta_{Y^{0}}(\SF\sus g) (\SF \delta) (\SF f),
\] 
which is equal to $(\sus'\SF g)\Theta_{X^{0}}(\SF\delta)(\SF f)$ since $(\sus'\SF g)\Theta_{X^{0}} =\Theta_{Y^{0}}(\SF\sus g)$ as $\Theta$ is natural. So \eqref{eqn:n+2-angulated-functor-implies-n-exangulated-functor-natural-transformation} commutes and $\Gamma$ is indeed a natural transformation.

Lastly, suppose $\fs(\delta)=[X^{\raisebox{0.5pt}{\scalebox{0.6}{$\bullet$}}}]$. That is, 
$\begin{tikzcd}[column sep=0.7cm]X^{0}\arrow{r}{d_{X}^{0}}&X^{1}\arrow{r}{d_{X}^{1}}&\cdots\arrow{r}{d_{X}^{n}}&X^{n+1}\arrow{r}{\delta}&\sus X^{0}\end{tikzcd}$ is an $(n+2)$-angle in $\CC$. Then, as $\SF$ is an $(n+2)$-angulated functor, 
\[
\begin{tikzcd}[column sep=1.7cm] \SF X^{0}\arrow{r}{\SF d_{X}^{0}}&\SF X^{1}\arrow{r}{\SF d_{X}^{1}}&\cdots\arrow{r}{\SF d_{X}^{n}}&\SF X^{n+1}\arrow{r}{\Theta_{X^{0}}\circ\SF \delta}&\sus' \SF X^{0}\end{tikzcd}
\]
is an $(n+2)$-angle in $\CC'$. 
Since $\Gamma_{(X^{n+1},X^{0})}(\delta) = \Theta_{X^{0}}\circ\SF \delta$, we see that 
$\fs'(\Gamma_{(X^{n+1},X^{0})}(\delta))=[\SF_{\com} X^{\raisebox{0.5pt}{\scalebox{0.6}{$\bullet$}}}]$.
Hence, $\SF$ satisfies Definition \ref{def:n-exangulated-functor}.

$(\Leftarrow)$ Conversely, assume $\SF$ satisfies Definition \ref{def:n-exangulated-functor} with natural transformation $\Gamma\colon \BE(-,-) \Rightarrow \BE'(\SF-,\SF-)$. 
Let $X\in\CC$ be arbitrary. 
Setting $X^{n+1} = \sus X$ and $X^{0}=X$, we obtain a morphism
\[
\begin{tikzcd}[column sep = 1.5cm]
\CC(\sus X,\sus X) = \BE(\sus X, X) \arrow{r}{\Gamma_{(\sus X,X)}} 
	& \BE'(\SF \sus X, \SF X) = \CC'(\SF\sus X,\sus' \SF X).
\end{tikzcd}
\]
Set $\Theta_{X}\deff \Gamma_{(\sus X,X)}(1_{\sus X})\colon \SF \sus X \to \sus'\SF X$. 

We claim that $\Theta = \{\Theta_{X}\}_{X\in\CC}$ is a natural isomorphism $\SF \sus\Rightarrow \sus'\SF$. 
Let $a\colon X\to X'$ be an arbitrary morphism in $\CC$. 
Then we must show that $(\sus'\SF a)\Theta_{X}= \Theta_{X'}(\SF\sus a)$. 
Since \eqref{eqn:n+2-angulated-functor-implies-n-exangulated-functor-natural-transformation} commutes for all $f\colon Y^{n+1}\to X^{n+1}$ and $g\colon X^{0}\to Y^{0}$ in $\CC$, we can first choose $f=1_{\sus X}$ and $g=a$, and then, using $\delta = 1_{\sus X}\in\BE(\sus X, X)$, we obtain $(\sus' \SF a)\Theta_{X}= \Gamma_{(\sus X,X')}(\sus a)$. 
Secondly, choosing $f=\sus a, g=1_{X'}$ and $\delta = 1_{\sus X'}\in\BE(\sus X',X')$ yields $\Gamma_{(\sus X, X')}(\sus a) = \Theta_{X'}(\SF\sus a)$. 
Consequently, $(\sus' \SF a)\Theta_{X} =\Theta_{X'}(\SF\sus a) $ and $\Theta$ is a natural transformation.

To show $\Theta_{X}$ is an isomorphism for each $X\in\CC$, consider the $(n+2)$-angle 
\[
\begin{tikzcd}
X \arrow{r}& 0 \arrow{r}& \cdots \arrow{r}& 0 \arrow{r}& \sus X \arrow{r}{1_{\sus X}}& \sus X
\end{tikzcd}
\] 
in $\CC$. 
Denote by $Z^{\raisebox{0.5pt}{\scalebox{0.6}{$\bullet$}}}$ the complex 
$\begin{tikzcd}[column sep=0.5cm]
X \arrow{r}& 0 \arrow{r}& \cdots \arrow{r}& 0 \arrow{r}& \sus X
\end{tikzcd}$ in $\com_{(X,\sus X)}^{n}$. 
Since $\SF$ is an $n$-exangulated functor and $\fs(1_{\sus X})= [Z^{\raisebox{0.5pt}{\scalebox{0.6}{$\bullet$}}}]$, we have that $\fs'(\Gamma_{(\sus X, X)}(1_{\sus X}))= [\SF_{\com} Z^{\raisebox{0.5pt}{\scalebox{0.6}{$\bullet$}}}]$. 
Hence, in $\CC'$ we have the $(n+2)$-angle 
\begin{equation}\label{eqn:delta-sub-sus-X}
\begin{tikzcd}
\SF X \arrow{r}& 0\arrow{r}& \cdots \arrow{r}& 0 \arrow{r}& \SF\sus X \arrow{r}{\Theta_{X}}& \sus' \SF X
\end{tikzcd}
\end{equation}
as $\Theta_{X}= \Gamma_{(\sus X,X)}(1_{\sus X})$. 
It follows that $\Theta_{X}$ is an isomorphism, because applying the functors $\CC'(\sus'\SF X,-)$ and $\CC'(-,\SF\sus X)$ to 
\eqref{eqn:delta-sub-sus-X} 
yields long exact sequences in $\Ab$.

Finally, suppose that $X^{\raisebox{0.5pt}{\scalebox{0.6}{$\bullet$}}}$ is an $(n+2)$-angle in $\CC$. Then 
\[
\begin{tikzcd}[column sep=2.6cm] \SF X^{0}\arrow{r}{\SF d_{X}^{0}}&\SF X^{1}\arrow{r}{\SF d_{X}^{1}}&\cdots\arrow{r}{\SF d_{X}^{n}}&\SF X^{n+1}
\arrow{r}{\Gamma_{(X^{n+1},X^{0})}(d_{X}^{n+1})}&\sus' \SF X^{0}\end{tikzcd}
\]
is an $(n+2)$-angle in $\CC'$. 
Note that the commutativity of \eqref{eqn:n+2-angulated-functor-implies-n-exangulated-functor-natural-transformation} with $f=d_{X}^{n+1}, g=1_{X^{0}}$ and $\delta = 1_{\sus X^{0}}\in \BE(\sus X^{0},X^{0})$ implies $\Gamma_{(X^{n+1},X^{0})}(d_{X}^{n+1})= \Gamma_{(\sus X^{0},X^{0})}(1_{\sus X^{0}})\circ \SF d_{X}^{n+1}=\Theta_{X^{0}}\circ \SF d_{X}^{n+1}$. Hence, 
\[
\begin{tikzcd}[column sep=2cm] \SF X^{0}\arrow{r}{\SF d_{X}^{0}}&\SF X^{1}\arrow{r}{\SF d_{X}^{1}}&\cdots\arrow{r}{\SF d_{X}^{n}}&\SF X^{n+1}
\arrow{r}{\Theta_{X^{0}}\circ \SF d_{X}^{n+1}}&\sus' \SF X^{0}\end{tikzcd}
\]
is an $(n+2)$-angle in $\CC'$.

\end{proof}

\begin{thm}
\label{thm:n-exact-functor-iff-n-exangulated}

Let $(\CC,\CX)$ and $(\CC',\CX')$ be $n$-exact categories. 
Following the notation of \emph{Example \ref{exa:n-exact-category-is-n-exangulated}}, set $\BE\deff\BE_{\CX}$ and $\BE'\deff\BE_{\CX'}$. 
Assume for all $X^{0},X^{n+1}\in\CC$ (respectively, $Z^{0},Z^{n+1}\in\CC'$) that $\BE(X^{n+1},X^{0})$ (respectively, $\BE'(Z^{n+1},Z^{0})$) is a set. 
Consider the $n$-exangulated categories  $(\CC,\BE,\fs)$ and $(\CC',\BE',\fs')$. 
Suppose $\SF\colon \CC\to\CC'$ is an additive covariant functor. 
Then $\SF$ is an $n$-exact functor if and only if $\SF$ is an $n$-exangulated functor.

\end{thm}

\begin{proof}

$(\Rightarrow)$ Suppose $\SF$ is an $n$-exact functor. For $X^{0},X^{n+1}$ in $\CC$, define $\Gamma_{(X^{n+1},X^{0})}\colon \BE(X^{n+1},X^{0})\to \BE'(\SF X^{n+1},\SF X^{0})$ by 
$\Gamma_{(X^{n+1},X^{0})}([X^{\raisebox{0.5pt}{\scalebox{0.6}{$\bullet$}}}])\deff [\SF_{\com} X^{\raisebox{0.5pt}{\scalebox{0.6}{$\bullet$}}}]$. 
Note that this is a well-defined map since the additive functor $\SF_{\com}$ preserves homotopies and since $\SF_{\com}(\CX)\subseteq \CX'$ as $\SF$ is $n$-exact.  
Now let us show that the collection of homomorphisms  $\Gamma=\{\Gamma_{(X^{n+1},X^{0})}\}_{(X^{n+1},X^{0})\in\CC^{\op}\times\CC}$ gives a natural transformation $\BE(-,-)\Rightarrow \BE'(\SF-,\SF-)$. Let $f\colon Y^{n+1}\to X^{n+1}$ and $g\colon X^{0}\to Y^{0}$ be arbitrary morphisms in $\CC$. 
Note that $\BE(f,g) = \BE(Y^{n+1},g)\circ\BE(f,X^{0})$ and $\BE'(\SF f,\SF g) = \BE'(\SF Y^{n+1},\SF g)\circ\BE'(\SF f,\SF X^{0})$, so it is enough to show that the squares 
\begin{equation}\label{eqn:n-exact-functor-implies-n-exangulated-functor-natural-transformation-for-f}
\begin{tikzcd}[column sep=2cm]
\BE(X^{n+1},X^{0}) \arrow{r}{\Gamma_{(X^{n+1},X^{0})}} \arrow{d}{\BE(f,X^{0})} & 
\BE'(\SF X^{n+1},\SF X^{0}) \arrow{d}{\BE'(\SF f, \SF X^{0})}\\
\BE(Y^{n+1},X^{0}) \arrow{r}{\Gamma_{(Y^{n+1},X^{0})}}&
\BE'(\SF Y^{n+1},\SF X^{0})
\end{tikzcd}\end{equation}
and 
\begin{equation}\label{eqn:n-exact-functor-implies-n-exangulated-functor-natural-transformation-for-g}
\begin{tikzcd}[column sep=2cm]
\BE(Y^{n+1},X^{0}) \arrow{r}{\Gamma_{(Y^{n+1},X^{0})}} \arrow{d}{\BE(Y^{n+1},g)} & 
\BE'(\SF Y^{n+1},\SF X^{0}) \arrow{d}{\BE'(\SF Y^{n+1},\SF g)}\\
\BE(Y^{n+1},Y^{0}) \arrow{r}{\Gamma_{(Y^{n+1},Y^{0})}}&
\BE'(\SF Y^{n+1},\SF X^{0})
\end{tikzcd}\end{equation}
commute.

We only show that \eqref{eqn:n-exact-functor-implies-n-exangulated-functor-natural-transformation-for-f} commutes as the commutativity of \eqref{eqn:n-exact-functor-implies-n-exangulated-functor-natural-transformation-for-g} can be shown similarly. 
Thus, let $[X^{\raisebox{0.5pt}{\scalebox{0.6}{$\bullet$}}}]\in\BE(X^{n+1},X^{0})$ be arbitrary. Then $\Gamma_{(X^{n+1},X^{0})}([X^{\raisebox{0.5pt}{\scalebox{0.6}{$\bullet$}}}])=[\SF_{\com} X^{\raisebox{0.5pt}{\scalebox{0.6}{$\bullet$}}}]$, and $\BE'(\SF f,\SF X^{0})([\SF_{\com} X^{\raisebox{0.5pt}{\scalebox{0.6}{$\bullet$}}}])=[Z^{\raisebox{0.5pt}{\scalebox{0.6}{$\bullet$}}}]$ is obtained by forming an $n$-pullback of $(\SF d_{X}^{1},\ldots,\SF d_{X}^{n})$ along $\SF f$, where $Z^{0}=\SF X^{0}$ and $Z^{n+1}=\SF Y^{n+1}$. 
By the dual of \cite[Prop.\ 4.8]{Jasso-n-abelian-and-n-exact-categories}, this yields a morphism $h^{\raisebox{0.5pt}{\scalebox{0.6}{$\bullet$}}}\colon Z^{\raisebox{0.5pt}{\scalebox{0.6}{$\bullet$}}}\to \SF_{\com} X^{\raisebox{0.5pt}{\scalebox{0.6}{$\bullet$}}}$ of $\CX'$-admissible $n$-exact sequences, where $h^{0}=1_{\SF X^{0}}$ and $h^{n+1}=\SF f$. 

On the other hand, $\BE(f,X^{0})([X^{\raisebox{0.5pt}{\scalebox{0.6}{$\bullet$}}}])=[W^{\raisebox{0.5pt}{\scalebox{0.6}{$\bullet$}}}]$ is given by taking an $n$-pullback of $(d_{X}^{1},\ldots,d_{X}^{n})$ along $f$, where $W^{0}=X^{0}$ and $W^{n+1}=Y^{n+1}$. 
By the dual of \cite[Prop.\ 4.8]{Jasso-n-abelian-and-n-exact-categories}, there is a morphism $f^{\raisebox{0.5pt}{\scalebox{0.6}{$\bullet$}}}\colon W^{\raisebox{0.5pt}{\scalebox{0.6}{$\bullet$}}}\to X^{\raisebox{0.5pt}{\scalebox{0.6}{$\bullet$}}}$ of $\CX$-admissible $n$-exact sequences such that $f^{0}=1_{X^{0}}$ and $f^{n+1}=f$. 
Note that $\Gamma_{(Y^{n+1},X^{0})}([W^{\raisebox{0.5pt}{\scalebox{0.6}{$\bullet$}}}])=[\SF_{\com} W^{\raisebox{0.5pt}{\scalebox{0.6}{$\bullet$}}}]$ and that, since $\SF$ is an $n$-exact functor, $\SF_{\com} f^{\raisebox{0.5pt}{\scalebox{0.6}{$\bullet$}}}\colon \SF_{\com} W^{\raisebox{0.5pt}{\scalebox{0.6}{$\bullet$}}}\to \SF_{\com} X^{\raisebox{0.5pt}{\scalebox{0.6}{$\bullet$}}}$ is a morphism of $\CX'$-admissible $n$-exact sequences.

By the dual of \cite[Prop.\ 4.9]{Jasso-n-abelian-and-n-exact-categories}, there exists a morphism $p^{\raisebox{0.5pt}{\scalebox{0.6}{$\bullet$}}}\colon Z^{\raisebox{0.5pt}{\scalebox{0.6}{$\bullet$}}}\to \SF_{\com} W^{\raisebox{0.5pt}{\scalebox{0.6}{$\bullet$}}}$ such that $p^{0}=h^{0}=1_{\SF X^{0}}$, and there exists a morphism $q^{\raisebox{0.5pt}{\scalebox{0.6}{$\bullet$}}}\colon \SF_{\com} W^{\raisebox{0.5pt}{\scalebox{0.6}{$\bullet$}}}\to Z^{\raisebox{0.5pt}{\scalebox{0.6}{$\bullet$}}}$ such that $q^{0}=\SF f^{0}=1_{\SF X^{0}}$. 
In particular, we then have a morphism $q^{\raisebox{0.5pt}{\scalebox{0.6}{$\bullet$}}}p^{\raisebox{0.5pt}{\scalebox{0.6}{$\bullet$}}}\colon Z^{\raisebox{0.5pt}{\scalebox{0.6}{$\bullet$}}}\to Z^{\raisebox{0.5pt}{\scalebox{0.6}{$\bullet$}}}$ with $q^{0}p^{0}=1_{\SF X^{0}}=1_{Z^{0}}$, 
and a morphism $p^{\raisebox{0.5pt}{\scalebox{0.6}{$\bullet$}}}q^{\raisebox{0.5pt}{\scalebox{0.6}{$\bullet$}}}\colon \SF_{\com} W^{\raisebox{0.5pt}{\scalebox{0.6}{$\bullet$}}}\to \SF_{\com} W^{\raisebox{0.5pt}{\scalebox{0.6}{$\bullet$}}}$ with $p^{0}q^{0}=1_{\SF X^{0}}=1_{\SF_{\com} W^{0}}$. 
Hence, by \cite[Lem.\ 2.1]{Jasso-n-abelian-and-n-exact-categories}, we have that $q^{\raisebox{0.5pt}{\scalebox{0.6}{$\bullet$}}}p^{\raisebox{0.5pt}{\scalebox{0.6}{$\bullet$}}}\sim 1_{Z^{\raisebox{0.5pt}{\scalebox{0.6}{$\bullet$}}}}$ and $p^{\raisebox{0.5pt}{\scalebox{0.6}{$\bullet$}}}q^{\raisebox{0.5pt}{\scalebox{0.6}{$\bullet$}}}\sim 1_{\SF_{\com} W^{\raisebox{0.5pt}{\scalebox{0.6}{$\bullet$}}}}$. 
Thus, the homotopy equivalence classes $[Z^{\raisebox{0.5pt}{\scalebox{0.6}{$\bullet$}}}]$ and $[\SF_{\com} W^{\raisebox{0.5pt}{\scalebox{0.6}{$\bullet$}}}]$ coincide, so \eqref{eqn:n-exact-functor-implies-n-exangulated-functor-natural-transformation-for-f} commutes and $\Gamma$ is a natural transformation.

Lastly, suppose $\fs(\delta)=[X^{\raisebox{0.5pt}{\scalebox{0.6}{$\bullet$}}}]$, which means $\delta=[X^{\raisebox{0.5pt}{\scalebox{0.6}{$\bullet$}}}]$ by Example \ref{exa:n-exact-category-is-n-exangulated}. 
As defined above, $\Gamma_{(X^{n+1},X^{0})}(\delta)=\Gamma_{(X^{n+1},X^{0})}([X^{\raisebox{0.5pt}{\scalebox{0.6}{$\bullet$}}}])= [\SF_{\com} X^{\raisebox{0.5pt}{\scalebox{0.6}{$\bullet$}}}]$ so 
$\fs'(\Gamma _{(X^{n+1},X^{0})}(\delta))=[\SF_{\com} X^{\raisebox{0.5pt}{\scalebox{0.6}{$\bullet$}}}]$.

$(\Leftarrow)$ Conversely, suppose $\SF$ is an $n$-exangulated functor. 
We must show that $\SF_{\com}(\CX)\subseteq\CX'$. 
Let 
\[
\begin{tikzcd}
X^{0}\arrow[tail]{r}{d^{0}}& X^{1} \arrow{r}{d^{1}}& \cdots\arrow{r}{d^{n-1}}&X^{n}\arrow[two heads]{r}{d^{n}}& X^{n+1}
\end{tikzcd}
\]
be an arbitrary $\CX$-admissible $n$-exact sequence. 
Set $\delta\deff[X^{\raisebox{0.5pt}{\scalebox{0.6}{$\bullet$}}}]\in\BE(X^{n+1},X^{0})$. 
Then $\fs(\delta)=[X^{\raisebox{0.5pt}{\scalebox{0.6}{$\bullet$}}}]$ implies $\fs'(\Gamma_{(X^{n+1},X^{0})}(\delta))=[\SF_{\com} X^{\raisebox{0.5pt}{\scalebox{0.6}{$\bullet$}}}]\in\BE'(\SF X^{n+1},\SF X^{0})$ as $\SF$ is $n$-exangulated. 
Hence, $\SF_{\com} X^{\raisebox{0.5pt}{\scalebox{0.6}{$\bullet$}}}$ is an $\CX'$-admissible $n$-exact sequence and we are done.

\end{proof}

We note here that the notion of an $n$-exangulated functor is currently being generalised. In ongoing with J.\ Haugland and M.\ H.\ Sand{\o{}}y, the authors are considering functors from an $n$-exangulated category to an $m$-exangulated category, which are compatible with the respective structures and where $n$ and $m$ might differ. See \cite{Bennett-TennenhausHauglandSandoyShah-the-category-of-extensions-and-n-m-exangulated-functors}.


\medskip
\section{Transport of structure}
\label{sec:transport-of-structure}

In \S\ref{sec:transport-of-structure}, let $n\geq 1$ be a positive integer. 
We show that an equivalence of categories preserves the structures discussed in \S\ref{sec:higher-structures}. In all cases, we will need the following well-known result, which is easily verified.

\begin{thm}
\label{thm:transport-of-additive-structure}

Let $\SF\colon \CC \to \CC'$ be an equivalence of categories and suppose $\CC$ is an additive category. Then $\CC'$ is an additive category and $\SF$ is an additive equivalence.

\end{thm}

For any category $\CA$, we will denote the \emph{identity functor} of $\CA$ by $\idfunc{\CA}$; see \cite[p.\ 6]{Borceux-handbook-1}.

\begin{setup}
\label{con:equivalence-unit-counit-setup}

Suppose $\SF\colon \CC\to \CC'$ is an equivalence. 
Then there exists a quasi-inverse $\SG\colon \CC'\to\CC$ for $\SF$. 
In this situation, 
the \emph{unit} $\Phi\colon \idfunc{\CC'} \Longrightarrow\SF\SG$ 
and 
the \emph{counit} $\Psi\colon \SG\SF{\Longrightarrow}\idfunc{\CC}$ are natural isomorphisms, and 
the \emph{counit-unit equations} (or \emph{triangle identities}) 
\begin{equation}
\label{eqn:triangle-identities}
	\SF\Psi\circ \Phi_{\SF}=1_{\SF} 
	\hspace{1cm}
	\text{and}
	\hspace{1cm} \Psi_{\SG}\circ \SG\Phi=1_{\SG},
\end{equation}
are satisfied; see \cite[Prop.\ 3.4.1]{Borceux-handbook-1}. 
The natural transformations $\SF\Psi, \Phi_{\SF}, \Psi_{\SG}$ and $\SG\Phi$ are obtained from \emph{whiskering}; 
see \cite[p.\ 24]{BaezBaratinFreidelWise-infinite-dimensional-representations-of-two-groups}. 

\end{setup}

We will refer to Setup \ref{con:equivalence-unit-counit-setup} in the main results of \S\S\ref{sec:transport-n-angulated-categories}--\ref{sec:transport-n-exangulated-categories}.


\smallskip
\subsection{\texorpdfstring{$(n+2)$}{n+2}-angulated categories}
\label{sec:transport-n-angulated-categories}

Recall that if $(\CC, \sus, \CT)$ is a weak $(n+2)$-angulated category, then in this article we only require $\sus$ to be an autoequivalence of $\CC$. 
Lemma \ref{lem:additive-category-with-autoequivalence-equivalent-to-n-angulated-category-is-also-n-angulated} is a key ingredient in the proof of Theorem \ref{thm:transport-of-n-angulated-structure}.

\begin{lem}
\label{lem:additive-category-with-autoequivalence-equivalent-to-n-angulated-category-is-also-n-angulated}

Let $(\CC, \sus, \CT)$ be a weak (pre-)$(n+2)$-angulated category and let $\CC'$ be a category equipped with an autoequivalence $\sus'$. 
Suppose $\SF\colon \CC \to \CC'$ is an equivalence and assume that there is a natural isomorphism $\Theta\colon \SF\sus \overset{\iso}{\Longrightarrow} \sus'\SF$. 
Let $\CT'$ be the collection of $(n+2)$-$\sus'$-sequences consisting of those  isomorphic to any $(n+2)$-$\sus'$-sequence of the form 
\begin{equation}\label{eqn:defining-n-angle-in-CC-prime}
\begin{tikzcd}[column sep=1.9cm]\SF X^{0}\arrow{r}{\SF d_{X}^{0}}&\SF X^{1}\arrow{r}{\SF d_{X}^{1}}&\cdots\arrow{r}{\SF d_{X}^{n}}&\SF X^{n+1}\arrow{r}{\Theta_{X^{0}}\circ\SF d_{X}^{n+1}}&\sus'\SF  X^{0},\end{tikzcd}\end{equation}
where $\begin{tikzcd}[column sep=0.7cm]X^{0}\arrow{r}{d_{X}^{0}}&X^{1}\arrow{r}{d_{X}^{1}}&\cdots\arrow{r}{d_{X}^{n}}&X^{n+1}\arrow{r}{d_{X}^{n+1}}&\sus X^{0}\end{tikzcd}$ is an $(n+2)$-angle in $\CC$. Then $(\CC', \sus', \CT')$ is a weak (pre-)$(n+2)$-angulated category and $\SF$ is an $(n+2)$-angle equivalence.

\end{lem}

\begin{proof}

Let us fix the notation as in Setup \ref{con:equivalence-unit-counit-setup}. 
By Theorem \ref{thm:transport-of-additive-structure}, we have that $\CC'$ is an additive category and $\SF$ is an additive equivalence. 
Note that this implies $\sus'$ is additive (also by Theorem \ref{thm:transport-of-additive-structure}). 
Suppose $(\CC,\sus,\CT)$ is a weak pre-$(n+2)$-angulated category, and let us first verify axioms \ref{F1}--\ref{F3} from Definition \ref{def:weak-pre-and-n-angulated-category} for $(\CC',\sus',\CT')$. 

One can easily verify that $\CT'$ is closed under taking finite direct sums and direct summands, using that $\SF$, $\sus$ and $\sus'$ are all additive functors and that $\CT$ satisfies \ref{F1}.

Let $X'\in\CC'$ be arbitrary. 
Then 
\[
\begin{tikzcd}[column sep=0.8cm]\SG X' \arrow{r}{1_{\SG X'}} &\SG X' \arrow{r} &0\arrow{r}&\cdots\arrow{r}&0\arrow{r} &\sus\SG X'\end{tikzcd}
\]
is an $(n+2)$-angle in $\CC$. 
Thus, 
\[
\begin{tikzcd}[column sep=1.2cm]
\SF\SG X' \arrow{r}{1_{\SF\SG X'}} &\SF\SG X' \arrow{r} &0\arrow{r}&\cdots\arrow{r}&0\arrow{r} &\sus'(\SF\SG X')\end{tikzcd}
\]
is an $(n+2)$-$\sus'$-sequence in $\CT'$. 
Using the natural isomorphism $\Phi\colon\idfunc{\CC'}\overset{\iso}{\Longrightarrow}\SF\SG$, we see that 
\[
\begin{tikzcd}[column sep=0.7cm]X' \arrow{r}{1_{X'}} &X' \arrow{r} &0\arrow{r}&\cdots\arrow{r}&0\arrow{r} &\sus'X'\end{tikzcd}
\]
is also an $(n+2)$-$\sus'$-sequence in $\CT'$.

Let $d_{X'}^{0}\colon X'^{0}\to X'^{1}$ be an arbitrary morphism in $\CC'$, and consider the morphism $\SG d_{X'}^{0}\colon \SG X'^{0} \to \SG X'^{1}$ in $\CC$. Using \ref{F1}(c) for $\CC$ we obtain an $(n+2)$-angle 
\[
\begin{tikzcd}
\SG X'^{0} \arrow{r}{\SG d_{X'}^{0}}& \SG X'^{1} \arrow{r}{d_{X}^{1}}& X^{2} \arrow{r}{d_{X}^{2}}& \cdots \arrow{r}{d_{X}^{n}}&X^{n+1} \arrow{r}{d_{X}^{n+1}}& \sus \SG X'^{0}.
\end{tikzcd}
\]
So, the $(n+2)$-$\sus'$-sequence 
\[
\hspace*{-0.15cm}
\begin{tikzcd}[column sep = 1.65cm]
\SF\SG X'^{0} \arrow{r}{\SF\SG d_{X'}^{0}}& \SF\SG X'^{1} \arrow{r}{\SF d_{X}^{1}}& \SF X^{2} \arrow{r}{\SF d_{X}^{2}}& \cdots \arrow{r}{\SF d_{X}^{n}}&\SF X^{n+1} \arrow{r}{\Theta_{\SG X'^{0}}\SF d_{X}^{n+1}}& \sus' \SF\SG X'^{0}.
\end{tikzcd}
\]
belongs to $\CT'$. 
Set $a\deff\Theta_{\SG X'^{0}}\SF d_{X}^{n+1}$. 
Then the diagram 
\[
\hspace*{-3pt}
\begin{tikzcd}[column sep = 1.6cm, row sep=1cm]
\SF\SG X'^{0} \arrow{r}{\SF\SG d_{X'}^{0}}\arrow{d}{\tensor[]{\Phi}{_{X'^{0}}^{-1}}}
& \SF\SG X'^{1} \arrow{r}{\SF d_{X}^{1}}\arrow{d}{\tensor[]{\Phi}{_{X'^{1}}^{-1}}}& \SF X^{2} \arrow{r}{\SF d_{X}^{2}}\arrow[equals]{d}& \cdots \arrow{r}{\SF d_{X}^{n}}&\SF X^{n+1} \arrow{r}{a}\arrow[equals]{d}& \sus' \SF\SG X'^{0}\arrow{d}{\tensor[]{\sus'\Phi}{_{X'^{0}}^{-1}}}\\
X'^{0} \arrow{r}{d_{X'}^{0}}& X'^{1} \arrow{r}{(\SF d_{X}^{1})\circ\Phi_{X'^{1}}}& \SF X^{2} \arrow{r}{\SF d_{X}^{2}}& \cdots \arrow{r}{\SF d_{X}^{n}}&\SF X^{n+1} \arrow{r}{(\tensor[]{\sus'\Phi}{_{X'^{0}}^{-1}})\circ a}& \sus'X'^{0}
\end{tikzcd}
\] 
commutes and all the vertical morphisms are isomorphisms, since $\Phi$ is a natural isomorphism. 
Thus the diagram is an isomorphism of $(n+2)$-$\sus'$-sequences and, hence, \ref{F1} holds for $(\CC',\sus',\CT')$.

By using suitable $(n+2)$-$\sus'$-sequence isomorphisms, we need only show the remaining axioms hold for the \emph{defining} $(n+2)$-$\sus'$-sequences of $\CT'$, i.e. $(n+2)$-$\sus'$-sequences of the form \eqref{eqn:defining-n-angle-in-CC-prime}, namely  
\[
\begin{tikzcd}[column sep=1.9cm]\SF X^{0}\arrow{r}{\SF d_{X}^{0}}&\SF X^{1}\arrow{r}{\SF d_{X}^{1}}&\cdots\arrow{r}{\SF d_{X}^{n}}&\SF X^{n+1}\arrow{r}{\Theta_{X^{0}}\circ\SF d_{X}^{n+1}}&\sus'\SF  X^{0}\end{tikzcd}
\] 
where $\begin{tikzcd}[column sep=0.7cm] X^{0}\arrow{r}{d_{X}^{0}}&X^{1}\arrow{r}{d_{X}^{1}}&\cdots\arrow{r}{d_{X}^{n}}&X^{n+1}\arrow{r}{d_{X}^{n+1}}&\sus X^{0}\end{tikzcd}$ is an $(n+2)$-angle in $\CC$. 
We may apply \ref{F2} in $\CC$ yielding the $(n+2)$-angle 
\[
\begin{tikzcd}[column sep=1.6cm] X^{1}\arrow{r}{d_{X}^{1}}&X^{2}\arrow{r}{d_{X}^{2}}&\cdots\arrow{r}{d_{X}^{n}}&X^{n+1}\arrow{r}{d_{X}^{n+1}}&\sus X^{0}\arrow{r}{(-1)^{n}\sus d_{X}^{0}}&\sus X^{1}\end{tikzcd}
\]
in $\CC$. 
Then, setting $b\deff (-1)^{n}\Theta_{X^{1}}\SF \sus d_{X}^{0}$, the $(n+2)$-angle 
\[
\begin{tikzcd}[column sep=1.2cm] \SF X^{1}\arrow{r}{\SF d_{X}^{1}}&\SF X^{2}\arrow{r}{\SF d_{X}^{2}}&\cdots\arrow{r}{\SF d_{X}^{n}}&\SF X^{n+1}\arrow{r}{\SF d_{X}^{n+1}}&\SF\sus  X^{0}\arrow{r}{b}&\sus'\SF  X^{1}\end{tikzcd}
\] 
is in $\CT'$. 
In the diagram
\[
\hspace*{-2pt}
\begin{tikzcd}[column sep=1.85cm] \SF X^{1}\arrow{r}{\SF d_{X}^{1}}\arrow[equals]{d}&\SF X^{2}\arrow{r}{\SF d_{X}^{2}}\arrow[equals]{d}&\cdots\arrow{r}{\SF d_{X}^{n}}&\SF X^{n+1}\arrow{r}{\SF d_{X}^{n+1}}\arrow[equals]{d}&
\SF\sus  X^{0}\arrow{r}{b}\arrow{d}{\Theta_{X^{0}}}&
\sus'\SF  X^{1}\arrow[equals]{d}\\
\SF X^{1}\arrow{r}{\SF d_{X}^{1}}&\SF X^{2}\arrow{r}{\SF d_{X}^{2}}&\cdots\arrow{r}{\SF d_{X}^{n}}&\SF X^{n+1}\arrow{r}{\Theta_{X^{0}}\SF d_{X}^{n+1}}&\sus'\SF  X^{0}\arrow{r}
{(-1)^{n} \sus'\SF d_{X}^{0}}&\sus'\SF  X^{1}\end{tikzcd}
\]
the rightmost square commutes since $\Theta$ is a natural isomorphism. 
Hence, the diagram above is an $(n+2)$-$\sus'$-sequence isomorphism and the bottom row is in $\CT'$. 
Thus, we conclude that if an $(n+2)$-$\sus'$-sequence is in $\CT'$, then its left rotation also lies in $\CT'$. The converse is similar, and so \ref{F2} holds for $(\CC',\sus',\CT')$.

In order to show \ref{F3} is satisfied, suppose we are given a commutative diagram 
\begin{equation}\label{eqn:F3-for-C-prime}
\begin{tikzcd}[column sep=2.1cm]
\SF X^{0} \arrow{r}{\SF d_{X}^{0}}\arrow{d}{f^{0}}\commutes{dr}&
\SF X^{1} \arrow{r}{\SF d_{X}^{1}}\arrow{d}{f^{1}}&\cdots\arrow{r}{\SF d_{X}^{n}}&
\SF X^{n+1}\arrow{r}{\Theta_{X^{0}}\circ(\SF d_{X}^{n+1})}&\sus'\SF X^{0}\arrow{d}{\sus'f^{0}} \\
\SF Y^{0} \arrow{r}{\SF d_{Y}^{0}}&
\SF Y^{1} \arrow{r}{\SF d_{Y}^{1}}&\cdots\arrow{r}{\SF d_{Y}^{n}}&
\SF Y^{n+1}\arrow{r}{\Theta_{Y^{0}}\circ(\SF d_{Y}^{n+1})}&\sus'\SF Y^{0},
\end{tikzcd}
\end{equation}
in which the rows are defining $(n+2)$-$\sus'$-sequences in $\CT'$ and $f^{1} \circ \SF d_{X}^{0} = (\SF d_{Y}^{0}) \circ f^{0}$. Since $\SF$ is full, we have that $f^{0} = \SF g^{0}$ and $f^{1} = \SF g^{1}$ for some $g^{0}\colon X^{0}\to Y^{0}$ and $g^{1}\colon X^{1}\to Y^{1}$. Therefore, the identity $f^{1} \circ \SF d_{X}^{0} = (\SF d_{Y}^{0}) \circ f^{0}$ becomes $\SF (g^{1}d_{X}^{0}) = \SF (d_{Y}^{0}g^{0})$, and we see that  $g^{1}d_{X}^{0} = d_{Y}^{0}g^{0}$ in $\CC$ as $\SF$ is also faithful. 
Hence, by \ref{F3} for $\CC$, we obtain a morphism
\begin{equation}\label{eqn:F3-for-C}
\begin{tikzcd}[column sep=1.6cm]
X^{0} \arrow{r}{d_{X}^{0}}\arrow{d}{g^{0}}\commutes{dr}&
X^{1} \arrow{r}{d_{X}^{1}}\arrow{d}{g^{1}}&
X^{2} \arrow{r}{d_{X}^{2}}\arrow[dotted]{d}{g^{2}}&\cdots\arrow{r}&
X^{n+1}\arrow{r}{d_{X}^{n+1}}\arrow[dotted]{d}{g^{n+1}}&
\sus X^{0}\arrow{d}{\sus g^{0}} \\
Y^{0} \arrow{r}{d_{Y}^{0}}&
Y^{1} \arrow{r}{d_{Y}^{1}}&
Y^{2} \arrow{r}{d_{Y}^{2}}&\cdots\arrow{r}&
Y^{n+1}\arrow{r}{d_{Y}^{n+1}}&\sus Y^{0}
\end{tikzcd}
\end{equation} 
of $(n+2)$-$\sus$-sequences. 
Applying $\SF_{\com}$ and using the natural isomorphism $\Theta$ to adjust the rightmost square, we have a morphism 
\[
\begin{tikzcd}[column sep=1.6cm]
\SF X^{0} \arrow{r}{\SF d_{X}^{0}}\arrow{d}{f^{0}} \commutes{dr}&
\SF X^{1} \arrow{r}{\SF d_{X}^{1}}\arrow{d}{f^{1}}&
\SF X^{2} \arrow{r}{\SF d_{X}^{2}}\arrow[dotted]{d}{f^{2}}&
\cdots\arrow{r}&
\SF X^{n+1}\arrow{r}{\Theta_{X^{0}}\SF d_{X}^{n+1}}\arrow[dotted]{d}{f^{n+1}}&
\sus'\SF X^{0}\arrow{d}{\sus'\SF g^{0}=\sus'f^{0}}\\
\SF Y^{0} \arrow{r}{\SF d_{Y}^{0}}&
\SF Y^{1} \arrow{r}{\SF d_{Y}^{1}}&
\SF Y^{2} \arrow{r}{\SF d_{Y}^{2}}&
\cdots\arrow{r}&
\SF Y^{n+1}\arrow{r}{\Theta_{Y^{0}}\SF d_{Y}^{n+1}}&
\sus'\SF Y^{0}
\end{tikzcd}
\] 
of $(n+2)$-$\sus'$-sequences in $\CC'$, where $f^{i}\deff\SF g^{i}$ for $2\leq i\leq n+1$, so \ref{F3} holds for $(\CC',\sus',\CT')$. 

Altogether, we have shown \ref{F1}--\ref{F3} hold for $(\CC',\sus',\CT')$ whenever these axioms hold for $(\CC,\sus,\CT)$. 
That is, if $\CC$ is weak pre-$(n+2)$-angulated then so too is $\CC'$. 

Suppose now that $(\CC,\sus,\CT)$ also satisfies \ref{F4}; we will show $(\CC',\sus',\CT')$ satisfies \ref{F4}. 
To this end, fix morphisms $f^{0}, f^{1}$ in $\CC'$ such that \eqref{eqn:F3-for-C-prime} commutes. 
As before, this induces diagram \eqref{eqn:F3-for-C} in $\CC$, and 
since \ref{F4} holds for $\CC$, we may assume that $g^{2},\ldots,g^{n+1}$  were obtained in such a way that the cone $C^{\raisebox{0.5pt}{\scalebox{0.6}{$\bullet$}}}\deff C(g)^{\raisebox{0.5pt}{\scalebox{0.6}{$\bullet$}}}$ 
\[
\begin{tikzcd}
X^{1}\oplus Y^{0} \arrow{r}{d_{C}^{0}}&X^{2}\oplus Y^{1}\arrow{r}{d_{C}^{1}}& \cdots\arrow{r}{d_{C}^{n}}&\sus X^{0}\oplus Y^{n+1} \arrow{r}{d_{C}^{n+1}}&\sus X^{1}\oplus \sus Y^{0}
\end{tikzcd}
\] 
lies in $\CT$. 
Therefore, 
\[
\begin{tikzcd}[column sep=1cm,ampersand replacement = \&]
\SF X^{1}\oplus \SF Y^{0} \arrow{r}{\SF d_{C}^{0}}\&
\cdots\arrow{r}{\SF d_{C}^{n}}\&
\SF \sus X^{0}\oplus \SF Y^{n+1} \arrow{r}{c}\&
\sus' \SF X^{1}\oplus \sus'\SF  Y^{0}
\end{tikzcd}
\]
lies in $\CT'$, 
where 
\[
c\deff 
\begin{pmatrix}
\Theta_{X^{1}} & 0 \\
0 & \Theta_{Y^{0}}
\end{pmatrix}
\SF d_{C}^{n+1}
= 
\begin{pmatrix}
-\Theta_{X^{1}}\SF\sus d_{X}^{0} & 0 \\
\Theta_{Y^{0}}\SF\sus g^{0} & \Theta_{Y^{0}}\SF d_{Y}^{n+1}
\end{pmatrix}.
\]
One can then check that the diagram 
\[
\hspace*{-2pt}
\begin{tikzcd}[column sep=1.15cm, row sep = 1cm, ampersand replacement = \&]
\SF X^{1}\oplus \SF Y^{0} \arrow{r}{\SF d_{C}^{0}}\arrow[equals]{d}\&
\cdots\arrow{r}{\SF d_{C}^{n-1}}\&
\SF X^{n+1}\oplus \SF Y^{n}\arrow{r}{\SF d_{C}^{n}}\arrow[equals]{d}\&
\SF \sus X^{0}\oplus \SF Y^{n+1} \arrow{r}{c}\arrow{d}{d}\&
\sus' \SF X^{1}\oplus \sus'\SF Y^{0}\arrow[equals]{d}\\
\SF X^{1}\oplus \SF Y^{0} \arrow{r}{\SF d_{C}^{0}}\&
\cdots\arrow{r}{\SF d_{C}^{n-1}}\&
\SF X^{n+1}\oplus \SF Y^{n}\arrow{r}{d\circ\SF d_{C}^{n}}\&
\sus'\SF X^{0}\oplus \SF Y^{n+1} \arrow{r}{e}\&
\sus' \SF X^{1}\oplus \sus'\SF  Y^{0},
\end{tikzcd}
\]
where 
\[
d\deff
\begin{pmatrix}
\Theta_{X^{0}}&0 \\ 0 & 1_{\SF Y^{n+1}}\end{pmatrix} 
\hspace{1cm}
\text{and}
\hspace{1cm}
e\deff \begin{pmatrix}-\sus'\SF d_{X}^{0} & 0 \\ 
\sus'\SF g^{0} & \Theta_{Y^{0}}\SF d_{Y}^{n+1}\end{pmatrix},
\] 
gives an isomorphism of $(n+2)$-$\sus'$-sequences using the natural isomorphism $\Theta$ for morphisms $d_{X}^{0}$ and $g^{0}$. 
Since the top row lies in $\CT'$ and the bottom row is precisely the cone $C(f)^{\raisebox{0.5pt}{\scalebox{0.6}{$\bullet$}}}$ of the morphism $(f^{0},\ldots,f^{n+1})$ in $\CC'$, we see $C(f)^{\raisebox{0.5pt}{\scalebox{0.6}{$\bullet$}}}$ belongs to $\CT'$. Thus, \ref{F4} holds for $(\CC',\sus',\CT')$ whenever it holds for $(\CC,\sus,\CT)$.

Lastly, we immediately see that $\SF$ is an $(n+2)$-angulated functor by the existence of the natural isomorphism $\Theta\colon \SF\sus \overset{\iso}{\Longrightarrow}\sus'\SF$ and by how we defined the (pre-)$(n+2)$-angulation $\CT'$. 
Furthermore, $\SF$ is thus an $(n+2)$-angle equivalence.

\end{proof}

We are now in a position to prove transport of structure for weak (pre-)$(n+2)$-angulated categories.

\begin{thm}
\label{thm:transport-of-n-angulated-structure}

Let $(\CC, \sus, \CT)$ be a weak (pre-)$(n+2)$-angulated category and suppose $\SF\colon \CC \to \CC'$ is an equivalence. 
Then $\CC'$ is a weak (pre-)$(n+2)$-angulated category and $\SF$ is an $(n+2)$-angle equivalence.

\end{thm}

\begin{proof}

Let us fix the notation as in Setup \ref{con:equivalence-unit-counit-setup}. 
By Theorem \ref{thm:transport-of-additive-structure}, we have that $\CC'$ is an additive category and $\SF$ is an additive equivalence of categories. 

Since $\sus$ is an autoequivalence of $\CC$, there exists a functor $\sus_{-}\colon \CC\to \CC$ and natural isomorphisms $\Gamma\colon\sus \sus_{-} \overset{\iso}{\Longrightarrow}\idfunc{\CC}$ and $\Lambda\colon\sus_{-} \sus \overset{\iso}{\Longrightarrow}\idfunc{\CC}$.  
Set $\sus'\deff\SF\sus\SG$ and $\sus'_{-}\deff\SF\sus_{-}\SG$. 
Then there is a natural isomorphism $\sus'\sus'_{-}\Rightarrow \idfunc{\CC'}$ given by the composition 
\[
\begin{tikzcd}[column sep=1.5cm]
\sus'\sus'_{-} 
= \SF\sus(\SG\SF)\sus_{-}\SG \arrow[Rightarrow]{r}{\SF\sus\Psi_{\sus_{-}\SG}} & 
\SF\sus\idfunc{\CC}\sus_{-}\SG 
= \SF\sus\sus_{-}\SG\arrow[Rightarrow]{r}{\SF\Gamma_{\SG}} & 
\SF\SG \arrow[Rightarrow]{r}{\Phi^{-1}} &  
\idfunc{\CC'}.
\end{tikzcd}
\] 
Similarly, there is also a natural isomorphism $\sus'_{-}\sus'\overset{\iso}{\Longrightarrow}\idfunc{\CC'}$, and hence $\sus'$ is an autoequivalence of $\CC'$.

Note that there is also a natural isomorphism
\[
\Theta\deff \SF\sus\Psi^{-1}\colon\SF\sus 
= \SF\sus\idfunc{\CC}\overset{\iso}{\Longrightarrow}\SF\sus\SG\SF 
= \sus'\SF.
\]
Then, by Lemma \ref{lem:additive-category-with-autoequivalence-equivalent-to-n-angulated-category-is-also-n-angulated}, $(\CC', \sus', \CT')$ is a weak (pre-)$(n+2)$-angulated category, where $\CT'$ is the collection of $(n+2)$-$\sus'$-sequences that are isomorphic to any $(n+2)$-$\sus'$-sequence of the form 
\[
\begin{tikzcd}[column sep=1.9cm]\SF X^{0}\arrow{r}{\SF d_{X}^{0}}&\SF X^{1}\arrow{r}{\SF d_{X}^{1}}&\cdots\arrow{r}{\SF d_{X}^{n}}&\SF X^{n+1}\arrow{r}{\Theta_{X^{0}}\circ\SF d_{X}^{n+1}}&\sus'\SF  X^{0},\end{tikzcd}
\]
such that $\begin{tikzcd}[column sep=0.7cm]X^{0}\arrow{r}{d_{X}^{0}}&X^{1}\arrow{r}{d_{X}^{1}}&\cdots\arrow{r}{d_{X}^{n}}&X^{n+1}\arrow{r}{d_{X}^{n+1}}&\sus X^{0}\end{tikzcd}$ is an $(n+2)$-angle in $\CC$. Furthermore, $\SF$ is a $(n+2)$-angle equivalence by the same result.

\end{proof}


\subsection{\texorpdfstring{$n$}{n}-exact and \texorpdfstring{$n$}{n}-abelian categories}
\label{sec:transport-n-exact-n-abelian-categories}

In \S\ref{sec:transport-n-exact-n-abelian-categories}, we show that whenever a category is equivalent to an $n$-exact (respectively, $n$-abelian) category, it must also be $n$-exact (respectively, $n$-abelian). The idea in both results is that the property of a complex becoming exact under a $\Hom$-functor is preserved by any fully faithful functor.

\begin{thm}
\label{thm:transport-of-n-exact-structure}

Let $(\CC, \CX)$ be an $n$-exact category and  suppose $\SF\colon \CC \to \CC'$ is an equivalence. Consider the class $\CX'$ consisting of all sequences in $\CC'$ that are weakly isomorphic to any sequence of the form
\begin{equation}\label{eqn:defining-n-exact-sequences-in-CC-prime}
\begin{tikzcd}\SF X^{0}\arrow{r}{\SF d_{X}^{0}}&\SF X^{1}\arrow{r}{\SF d_{X}^{1}}&\cdots\arrow{r}{\SF d_{X}^{n}}&\SF X^{n+1},\end{tikzcd}\end{equation}
where $\begin{tikzcd}
X^{0}\arrow[tail]{r}{d_{X}^{0}}&X^{1}\arrow{r}{d_{X}^{1}}&\cdots\arrow[two heads]{r}{d_{X}^{n}}&X^{n+1}\end{tikzcd}$ is an $\CX$-admissible $n$-exact sequence. 
Then $(\CC',\CX')$ forms an $n$-exact category and $\SF$ is an $n$-exact equivalence.

\end{thm}

\begin{proof}

First, note that $\CX'$ does indeed consist of $n$-exact sequences in $\CC'$, because $\SF$ is an equivalence of categories and so gives bijections on morphism sets. Second, by definition $\CX'$ is closed under weak isomorphisms. Now let us show the axioms for an $n$-exact category are satisfied by the pair $(\CC',\CX')$.

Since $0\rightarrowtail 0 \to \cdots \to 0\onto 0$ is an $\CX$-admissible $n$-exact sequence and $\SF$ is an additive functor, the sequence $0\to 0\to \cdots \to 0\to 0$ is in $\CX'$, so \ref{nE0} holds for $(\CC',\CX')$.

With the use of suitable isomorphisms of $n$-exact sequences, we are reduced to showing the remaining axioms for the \emph{defining} sequences of the form \eqref{eqn:defining-n-exact-sequences-in-CC-prime}. 
Thus, suppose we have two composable $\CX'$-admissible monomorphisms  $\SF d_{X}^{0}\colon \SF X^{0}\to \SF X^{1}$ and $\SF d_{Y}^{0}\colon \SF X^{1}\to \SF Y^{1}$, where $d_{Y}^{0}$ and $d_{X}^{0}$ are $\CX$-admissible monomorphisms. 
Then the morphism $d_{Y}^{0} d_{X}^{0}$ is an $\CX$-admissible monomorphism as $(\CC,\CX)$ is an $n$-exact category. 
Hence,  $\SF d_{Y}^{0}\circ \SF d_{X}^{0}=\SF(d_{Y}^{0}d_{X}^{0})$ is an $\CX'$-admissible monomorphism and $(\CC',\CX')$ satisfies \ref{nE1}. 
Dually, $(n$-E$1^{\op})$ is also satisfied.

Suppose we have a defining $\CX'$-admissible $n$-exact sequence of the form \eqref{eqn:defining-n-exact-sequences-in-CC-prime} 
and an arbitrary morphism $f^{0}\colon \SF X^{0}\to Y'$ in $\CC'$. 
Since $\SF$ is essentially surjective, there exists $Y^{0}$ in $\CC$ and an isomorphism $h\colon Y'\to \SF Y^{0}$. 
Then there exists a morphism $g^{0}\colon X^{0}\to Y^{0}$ such that $\SF g^{0}=hf^{0}$ as $\SF$ is full. 
Applying \ref{nE2} in $\CC$, we obtain an $n$-pushout
\begin{equation}\label{eqn:n-pushout-of-X-along-g}
\begin{tikzcd}
X^{0}\arrow[tail]{r}{d^{0}_{X}}\arrow{d}{g^{0}}& X^{1} \arrow{r}{d^{1}_{X}}\arrow[dotted]{d}{g^{1}}& \cdots\arrow{r}{d^{n-1}_{X}}&X^{n}\arrow[dotted]{d}{g^{n}}\\
Y^{0}\arrow[dotted]{r}[swap]{d^{0}_{Y}}& Y^{1} \arrow[dotted]{r}[swap]{d^{1}_{Y}}& \cdots\arrow[dotted]{r}[swap]{d^{n-1}_{Y}}&Y^{n},
\end{tikzcd}\end{equation}
such that $d^{0}_{Y}$ is an $\CX$-admissible monomorphism.

Applying $\SF$ yields a morphism
\begin{equation}\label{eqn:n-pushout-of-FX-along-Fg}
\begin{tikzcd}[column sep=1.3cm]
\SF X^{0}\arrow{r}{\SF d^{0}_{X}}\arrow{d}[swap]{\SF g^{0}=hf^{0}}& \SF X^{1} \arrow{r}{\SF d^{1}_{X}}\arrow{d}{\SF g^{1}}& \cdots\arrow{r}{\SF d^{n-1}_{X}}&\SF X^{n}\arrow{d}{\SF g^{n}}\\
\SF Y^{0}\arrow{r}[swap]{\SF d^{0}_{Y}}& \SF Y^{1} \arrow{r}[swap]{\SF d^{1}_{Y}}& \cdots\arrow{r}[swap]{\SF d^{n-1}_{Y}}&\SF Y^{n}
\end{tikzcd}\end{equation}
of $\CX'$-admissible $n$-exact sequences by definition of $\CX'$. We claim that this diagram is an $n$-pushout in $\CC'$. 
As \eqref{eqn:n-pushout-of-X-along-g} is an $n$-pushout, we know that in the mapping cone $C^{\raisebox{0.5pt}{\scalebox{0.6}{$\bullet$}}}\deff\cone(g)^{\raisebox{0.5pt}{\scalebox{0.6}{$\bullet$}}}$, the sequence $(d_{C}^{0},\ldots, d_{C}^{n-1})$ is an $n$-cokernel for $d_{C}^{-1}$. 
Thus, for any $Z\in \CC$, the induced commutative diagram 
\[
\begin{tikzcd}[column sep=1.4cm, row sep=1cm]
0\arrow{r}\arrow[equals]{d}&
\CC(Y^{n},Z)\arrow{r}{-\circ d_{C}^{n-1}}\arrow{d}{\SF_{Y^{n},Z}}[swap]{\iso}&
\cdots
\arrow{r}{-\circ d_{C}^{0}}&
\CC(X^{1}\oplus Y^{0},Z)\arrow{r}{-\circ d_{C}^{-1}} \arrow{d}{\SF_{X^{1}\oplus Y^{0},Z}}[swap]{\iso}&
\CC(X^{0},Z)\arrow{d}{\SF_{X^{0},Z}}[swap]{\iso}\\
0\arrow{r}&
\CC'(\SF Y^{n},\SF Z)\arrow{r}{-\circ \SF d_{C}^{n-1}}&
\cdots
\arrow{r}{-\circ \SF d_{C}^{0}}&
\CC'(\SF (X^{1}\oplus Y^{0}),\SF Z)\arrow{r}{-\circ \SF d_{C}^{-1}} &
\CC'(\SF X^{0},\SF Z)
\end{tikzcd}
\]
in $\Ab$ has an exact top row. 
The vertical homomorphisms are isomorphisms because $\SF$ is fully faithful, which thus implies exactness of the bottom row. As $\SF$ is essentially surjective, it follows that the sequence $(d_{D}^{0},\ldots, d_{D}^{n-1})$ is an $n$-cokernel for $d_{D}^{-1}$ in the mapping cone $D^{\raisebox{0.5pt}{\scalebox{0.6}{$\bullet$}}}\deff\cone(\SF_{\com} g)^{\raisebox{0.5pt}{\scalebox{0.6}{$\bullet$}}}$. Hence, \eqref{eqn:n-pushout-of-FX-along-Fg} is an $n$-pushout, as claimed.

Since $h\colon Y'\to \SF Y^{0}$ is an isomorphism, we can adjust \eqref{eqn:n-pushout-of-FX-along-Fg} to get another $n$-pushout
\[
\begin{tikzcd}[column sep=1.3cm]
\SF X^{0}\arrow{r}{\SF d^{0}_{X}}\arrow{d}{f^{0}}& \SF X^{1} \arrow{r}{\SF d^{1}_{X}}\arrow[dotted]{d}{\SF g^{1}}& \cdots\arrow{r}{\SF d^{n-1}_{X}}&\SF X^{n}\arrow[dotted]{d}{\SF g^{n}}\\
Y'\arrow[dotted]{r}[swap]{(\SF d^{0}_{Y})h}& \SF Y^{1} \arrow[dotted]{r}[swap]{\SF d^{1}_{Y}}& \cdots\arrow[dotted]{r}[swap]{\SF d^{n-1}_{Y}}&\SF Y^{n},
\end{tikzcd}
\]
where the bottom row is isomorphic to $\SF_{\com} Y^{\raisebox{0.5pt}{\scalebox{0.6}{$\bullet$}}}$. 
In particular, $(\SF d^{0}_{Y})h$ is an $\CX'$-admissible monomorphism. 
Thus, \ref{nE2} holds in $(\CC',\CX')$. Dually, $(n$-E$2^{\op})$ is satisfied.

Hence, $(\CC',\CX')$ is an $n$-exact category. The final claim follows from the definition of $\CX'$.

\end{proof}

Using similar methods as in the proof of Theorem \ref{thm:transport-of-n-exact-structure}, one may also show the following result, which we state without proof.

\begin{thm}
\label{thm:transport-of-n-abelian-structure}

Let $\CC$ be an $n$-abelian category and suppose $\SF\colon \CC \to \CC'$ is an equivalence. Then $\CC'$ is an $n$-abelian category and $\SF$ is an $n$-exact equivalence.

\end{thm}


\smallskip
\subsection{\texorpdfstring{$n$}{n}-exangulated categories}
\label{sec:transport-n-exangulated-categories}

The goal of \S\ref{sec:transport-n-exangulated-categories} is to show that one can transport an $n$-exangulated structure across an equivalence.

\begin{rem}
\label{rem:E-notation-changes}

We now supplement the discussion after Definition \ref{def:n-exangles}. Suppose $\SF\colon \CC \to \CC'$ is a functor and $\BE\colon \CC^{\op}\times\CC\to\Ab$ and $\BE'\colon \CC'^{\op}\times\CC'\to\Ab$ are biadditive functors. In Lemma \ref{lem:transport-of-exact-realisation} and Theorem \ref{thm:transport-of-n-exangulated-structure} below, it is necessary to distinguish  between $\BE$-extensions and $\BE'$-extensions, and describe relationships between them. This necessity motivates the use of the symbol $\BE$ in denoting $\delta_{\sharp}^{\BE}$, $\delta^{\sharp}_{\BE}$, $x_{\BE}\delta$  and $z^{\BE}\delta$. Note that these were denoted $\delta_{\sharp}$, $\delta^{\sharp}$, $x_{*}\delta$  and $z^{*}\delta$, respectively in 
\cite{HerschendLiuNakaoka-n-exangulated-categories-I-definitions-and-fundamental-properties} 
and 
\cite{NakaokaPalu-extriangulated-categories-hovey-twin-cotorsion-pairs-and-model-structures}.

\end{rem}

\begin{lem}
\label{lem:transport-of-exact-realisation}

Let $(\CC,\BE,\fs)$ be an $n$-exangulated category and suppose $\SF\colon \CC \to \CC'$ is an equivalence. 
Fix the notation as in \emph{Setup \ref{con:equivalence-unit-counit-setup}}. 
Define the functor $\BE'\colon \CC'^{\op}\times\CC'\to\Ab$ to be $\BE(\SG-,\SG-)$. 
For each $A,C\in\CC'$ and each $\delta\in\BE'(C,A)=\BE(\SG C,\SG A)$, set $\fs'(\delta)$ to be the homotopy equivalence class 
\[
[
	\begin{tikzcd}[column sep=1.9cm]
	A \arrow{r}{(\SF d_{X}^{0})\Phi_{A}} 
		& \SF X^{1} \arrow{r}{\SF d_{X}^{1}}
		& \cdots \arrow{r}{\SF d_{X}^{n-1}}
		& \SF X^{n} \arrow{r}{\tensor[]{\Phi}{_{C}^{-1}}\circ\SF d_{X}^{n}}
		&C
	\end{tikzcd}
]
\]
in $\kom_{(A,C)}^{n}$, 
where 
$\fs(\delta)=
[
\begin{tikzcd}
\SG A \arrow{r}{d_{X}^{0}}& X^{1} \arrow{r}{d_{X}^{1}}& \cdots \arrow{r}{d_{X}^{n-1}}& X^{n} \arrow{r}{d_{X}^{n}}&\SG C
\end{tikzcd}
]$. 
Then $\BE'$ is a biadditive functor and $\fs'$ is an exact realisation of $\BE'$.

\end{lem}

\begin{proof}

It is straightforward to show that $\BE'\colon \CC'^{\op}\times\CC'\to\Ab$ is a biadditive functor. 
Using that $\SF$ is an equivalence, and in particular that $\Phi$ and $\Psi$ are natural isomorphisms, one can also check that $\fs'$ is a well-defined assignment. 

Let us show that $\fs'$ is a realisation of $\BE'$. 
Let $A,B,C,D\in\CC'$ and suppose we have extensions $\delta\in\BE'(C,A)$ and $\eps\in\BE'(D,B)$. Suppose 
\[
\fs'(\delta)=
[\begin{tikzcd}[column sep=1.6cm]
A \arrow{r}{(\SF d_{X}^{0})\Phi_{A}}& \SF X^{1} \arrow{r}{\SF d_{X}^{1}}& \cdots \arrow{r}{\SF d_{X}^{n-1}}& \SF X^{n} \arrow{r}{\tensor[]{\Phi}{_{C}^{-1}}\SF d_{X}^{n}}&C
\end{tikzcd}]
\] 
and 
\[
\fs'(\eps)=
[\begin{tikzcd}[column sep=1.6cm]
B \arrow{r}{(\SF d_{Y}^{0})\Phi_{B}}& \SF Y^{1} \arrow{r}{\SF d_{Y}^{1}}& \cdots \arrow{r}{\SF d_{Y}^{n-1}}& \SF Y^{n} \arrow{r}{\tensor[]{\Phi}{_{D}^{-1}}\SF d_{Y}^{n}}& D
\end{tikzcd}],
\] 
where 
\[
\fs(\delta)=
[
\begin{tikzcd}
\SG A \arrow{r}{d_{X}^{0}}& X^{1} \arrow{r}{d_{X}^{1}}& \cdots \arrow{r}{d_{X}^{n-1}}& X^{n} \arrow{r}{d_{X}^{n}}&\SG C
\end{tikzcd}
]
\] 
and 
\[\fs(\eps)=
[
\begin{tikzcd}
\SG B \arrow{r}{d_{Y}^{0}}& Y^{1} \arrow{r}{d_{Y}^{1}}& \cdots \arrow{r}{d_{Y}^{n-1}}& Y^{n} \arrow{r}{d_{Y}^{n}}&\SG D
\end{tikzcd}
].
\] 
Let $(a,c)\colon \delta \to \eps$ be a morphism of $\BE'$-extensions, where $a\colon A\to B$ and $c\colon C\to D$ are morphisms in $\CC'$. 
Then we have 
\begin{align*}
(\SG a)_{\BE}(\delta)	
	& = \BE(\SG C,\SG a)(\delta)\\
     	& = \BE'(C,a)(\delta)\\
     	& = a_{\BE'}\delta \\
	& = c^{\BE'}\delta\\
	& = \BE'(c,B)(\eps)\\
	& = \BE(\SG c,\SG B)(\eps)\\
	& = (\SG c)^{\BE}(\eps),
\end{align*}
which implies $(\SG a,\SG c)\colon\delta\to\eps$ is a morphism of $\BE$-extensions. 
Thus, since $\fs$ is a realisation, there is a morphism $f^{\raisebox{0.5pt}{\scalebox{0.6}{$\bullet$}}}=(\SG a,f^{1},\ldots,f^{n},\SG c)$ of complexes in $\com^{n}_{\CC}$ depicted by 
\begin{equation}\label{eqn:realisation-of-morphism-a-c-of-BE-extensions}
\begin{tikzcd}
\SG A \arrow{r}{d_{X}^{0}}\arrow{d}{\SG a} & X^{1} \arrow{r}{d_{X}^{1}} \arrow{d}{f^{1}}& \cdots \arrow{r}{d_{X}^{n-1}}& X^{n} \arrow{r}{d_{X}^{n}} \arrow{d}{f^{n}}&\SG C \arrow{d}{\SG c}\\
\SG B \arrow{r}{d_{Y}^{0}}& Y^{1} \arrow{r}{d_{Y}^{1}}& \cdots \arrow{r}{d_{Y}^{n-1}}& Y^{n} \arrow{r}{d_{Y}^{n}}&\SG D.
\end{tikzcd}
\end{equation}

There are commutative squares 
\begin{equation}\label{eqn:comm-sqs-from-natural-iso-a-c}
\begin{tikzcd}[row sep=0.8cm,column sep=1.3cm]
A \arrow{r}{\Phi_{A}}\arrow{d}{a}& \SF\SG A \arrow{d}{\SF\SG a}\\
B \arrow{r}{\Phi_{B}}& \SF\SG B
\end{tikzcd}
\hspace{1cm}
\text{and}
\hspace{1cm}
\begin{tikzcd}[row sep=0.8cm,column sep=1.3cm]
\SF\SG C \arrow{r}{\tensor[]{\Phi}{_{C}^{-1}}}\arrow{d}{\SF\SG c}& C\arrow{d}{c} \\
\SF\SG D \arrow{r}{\tensor[]{\Phi}{_{D}^{-1}}}& D
\end{tikzcd}
\end{equation} 
from the natural isomorphism $\Phi\colon \idfunc{\CC'}\Rightarrow \SF\SG$. 
Thus, applying $\SF$ to \eqref{eqn:realisation-of-morphism-a-c-of-BE-extensions} and adjusting the resultant diagram with the squares in \eqref{eqn:comm-sqs-from-natural-iso-a-c}, we obtain a commutative diagram 
\begin{equation}
\begin{tikzcd}[column sep=1.8cm, row sep=1.2cm]
A \arrow{r}{(\SF d_{X}^{0})\Phi_{A}}\arrow{d}{a} & \SF X^{1} \arrow{r}{\SF d_{X}^{1}} \arrow{d}{\SF f^{1}}& \cdots \arrow{r}{\SF d_{X}^{n-1}}& \SF X^{n} \arrow{r}{\tensor[]{\Phi}{_{C}^{-1}}\SF d_{X}^{n}} \arrow{d}{\SF f^{n}}&C \arrow{d}{c}\\
B \arrow{r}{(\SF d_{Y}^{0})\Phi_{B}}& \SF Y^{1} \arrow{r}{\SF d_{Y}^{1}}& \cdots \arrow{r}{\SF d_{Y}^{n-1}}& \SF Y^{n} \arrow{r}{\tensor[]{\Phi}{_{D}^{-1}}\SF d_{Y}^{n}}&D
\end{tikzcd}
\end{equation} 
in $\CC'$. 
That is, $(a,\SF f^{1},\ldots,\SF f^{n},c)$ is a morphism in $\com_{\CC'}^{n}$ and hence $\fs'$ is a realisation of $\BE'$.

Now let us show that $\fs'$ satisfies \ref{R1}. Suppose that 
\[
\fs'(\delta)=
[\begin{tikzcd}[column sep=1.7cm]
A \arrow{r}{(\SF d_{X}^{0})\Phi_{A}}& \SF X^{1} \arrow{r}{\SF d_{X}^{1}}& \cdots \arrow{r}{\SF d_{X}^{n-1}}& \SF X^{n} \arrow{r}{\tensor[]{\Phi}{_{C}^{-1}}\SF d_{X}^{n}}&C
\end{tikzcd}],
\] 
where 
$\fs(\delta)=
[
\begin{tikzcd}
\SG A \arrow{r}{d_{X}^{0}}& X^{1} \arrow{r}{d_{X}^{1}}& \cdots \arrow{r}{d_{X}^{n-1}}& X^{n} \arrow{r}{d_{X}^{n}}&\SG C
\end{tikzcd}
].$ 
Let $Z\in\CC'$ be arbitrary, and consider the following diagram 
\[
\begin{tikzcd}[column sep=1.3cm, row sep=1.2cm]
\CC(\SG Z,\SG A) \arrow{r}{d_{X}^{0}\circ-}\arrow{d}{\SF_{\SG Z,\SG A }}&\CC(\SG Z,X^{1}) \arrow{r}{d_{X}^{1}\circ-}\arrow{d}{\SF_{\SG Z, X^{1}}}& \cdots \arrow{r}{d_{X}^{n}\circ-}& \CC(\SG Z,\SG C) \arrow{r}{(\delta_{\sharp}^{\BE})_{\SG Z}}\arrow{d}{\SF_{\SG Z, \SG C}}&\BE(\SG Z,\SG A)\arrow[equals]{dd}\\
\CC'(\SF \SG Z,\SF \SG A) \arrow{r}{\SF d_{X}^{0}\circ-}\arrow{d}{g^{0}}&\CC'(\SF \SG Z,\SF X^{1}) \arrow{r}{\SF d_{X}^{1}\circ-}\arrow{d}{g^{1}}& \cdots \arrow{r}{\SF d_{X}^{n}\circ-}& \CC'(\SF \SG Z,\SF \SG C)\arrow{d}{g^{n+1}}&\\
\CC'(Z,A) \arrow{r}{(\SF d_{X}^{0})\Phi_{A}\circ-}&\CC'(Z,\SF X^{1}) \arrow{r}{\SF d_{X}^{1}\circ-}& \cdots \arrow{r}{\tensor[]{\Phi}{_{C}^{-1}}\SF d_{X}^{n}\circ-}& \CC'(Z,C) \arrow{r}{(\delta_{\sharp}^{\BE'})_{Z}}&\BE'(Z,A),
\end{tikzcd}
\] 
where the upper vertical maps are those induced by the functor $\SF$, and 
$g^{0}=\tensor[]{\Phi}{_{A}^{-1}}\circ - \circ \Phi_{Z}$, 
$g^{i}= - \circ \Phi_{Z}$ for each $i=1,\ldots,n$, 
and 
$g^{n+1}=\tensor[]{\Phi}{_{C}^{-1}}\circ - \circ \Phi_{Z}$. 
It is easy to check that the whole diagram commutes, except possibly the far right rectangle. To this end, suppose $v\colon \SG Z \to \SG C$ is arbitrary. As $\SG$ is full, there exists $u\colon Z\to C$ such that $\SG u=v$. Since $\Phi\colon \idfunc{\CC'}\Rightarrow\SF\SG$ is a natural isomorphism, we have that $u=\tensor[]{\Phi}{_{C}^{-1}}(\SF\SG u)\Phi_{Z}$. 
Then we see that 
\begin{align*}
(\delta_{\sharp}^{\BE'})_{Z}(g^{n+1}(\SF_{\SG Z,\SG C}(v))) 
	&= (\delta_{\sharp}^{\BE'})_{Z}(\tensor[]{\Phi}{_{C}^{-1}}(\SF\SG u)\Phi_{Z})\\
	&= (\delta_{\sharp}^{\BE'})_{Z}(u) \\
	&= \BE'(u,A)(\delta)\\
	&= \BE(\SG u, \SG A)(\delta)\\
	&= \BE(v, \SG A)(\delta)\\
	&= (\delta_{\sharp}^{\BE})_{\SG Z}(v).
\end{align*}
Note that all the vertical maps are bijections, because $\SF$ is fully faithful and $\Phi_{X'}$ is an isomorphism for all $X'\in\CC'$. 
Since $\fs$ is an exact realisation of $\BE$, the top row is exact and hence so is the bottom row.

Similarly, one can show that there is an exact sequence
\[]
\begin{tikzcd}[column sep=2.15cm]
\CC'(C,Z) \arrow{r}{-\circ\tensor[]{\Phi}{_{C}^{-1}}\SF d_{X}^{n}} & \CC'(\SF X^{n},Z) \arrow{r}{-\circ\SF d_{X}^{n-1}}& \cdots \arrow{r}{-\circ(\SF d_{X}^{0})\Phi_{A}}& \CC'(A,Z) \arrow{r}{(\delta^{\sharp}_{\BE'})_{Z}}&\BE'(C,Z),
\end{tikzcd}
\]
and thus $\lan\begin{tikzcd}[column sep=1.7cm]
A \arrow{r}{(\SF d_{X}^{0})\Phi_{A}}& \SF X^{1} \arrow{r}{\SF d_{X}^{1}}& \cdots \arrow{r}{\SF d_{X}^{n-1}}& \SF X^{n} \arrow{r}{\tensor[]{\Phi}{_{C}^{-1}}\SF d_{X}^{n}}&C
\end{tikzcd} , \delta\ran$ is an $n$-exangle.

Lastly, we show that \ref{R2} holds for $\fs'$. Let $A\in\CC'$ be arbitrary. Then, as $\fs$ is an exact realisation of $\BE$, we have 
\[
\fs(_{\SG A}0_{\SG 0})=[
\begin{tikzcd}
\SG A \arrow{r}{1_{\SG A}}& \SG A\arrow{r}{} & 0 \arrow{r}{}& \cdots \arrow{r}{}& \SG 0
\end{tikzcd}].
\] 
Note that the element $_{\SG A}0_{\SG 0}\in\BE(\SG 0, \SG A)$ is the trivial element $_{A}0_{0}$ in $\BE'(0,A)$, so by definition 
$\fs'(_{A}0_{0})=\fs'(_{\SG A}0_{\SG 0})=[
\begin{tikzcd}
A \arrow{r}{\Phi_{A}}& \SF \SG A\arrow{r}{} & 0 \arrow{r}{}& \cdots \arrow{r}{}& 0
\end{tikzcd}].$ 
Lastly, we see that 
\[
\fs'(_{A}0_{0})=[
\begin{tikzcd}
A \arrow{r}{1_{A}}& A\arrow{r}{} & 0 \arrow{r}{}& \cdots \arrow{r}{}& 0
\end{tikzcd}],
\] 
using the isomorphism  $(1_{A},\Phi_{A},0,\ldots,0)$ in $\com_{(A,0)}^{n}$. 
Similarly, one can also show that 
\[
\fs'(_{0}0_{A})=[
\begin{tikzcd}
0 \arrow{r}{}& \cdots \arrow{r}{}& 0 \arrow{r}{}& A \arrow{r}{1_{A}}& A
\end{tikzcd}]
\] 
for the trivial element $_{0}0_{A}\in\BE'(A,0)$. 
Therefore, $\fs'$ is an exact realisation of $\BE'$.

\end{proof}

\begin{thm}
\label{thm:transport-of-n-exangulated-structure}

Let $(\CC,\BE,\fs)$ be an $n$-exangulated category and suppose $\SF\colon \CC \to \CC'$ is an equivalence.  
Then there is a biadditive functor $\BE'\colon\CC'^{\op}\times\CC'\to\Ab$ and an exact realisation $\fs'$ of $\BE'$, such that $(\CC',\BE',\fs')$ is an $n$-exangulated category and $\SF$ is an $n$-exangle equivalence.

\end{thm}

\begin{proof}

Fix the notation as in Setup \ref{con:equivalence-unit-counit-setup}. Let $\BE'=\BE(\SG-,\SG-)\colon \CC'^{\op}\times\CC'\to\Ab$ be the functor and $\fs'$ the assignment defined in Lemma \ref{lem:transport-of-exact-realisation}. From Lemma \ref{lem:transport-of-exact-realisation}, we have that $\fs'$ is an exact realisation of the biadditive functor $\BE'$. 
It remains to show that the triple $(\CC',\BE',\fs')$ satisfies axioms \ref{nEA1}, \ref{nEA2} and ($n$-EA$2^{\op}$).

For \ref{nEA1}, suppose we have morphisms $a\colon A\to B$ and $b\colon B\to C$ in $\CC'$ that are $\fs'$-inflations. 
That is, there are complexes $U^{\raisebox{0.5pt}{\scalebox{0.6}{$\bullet$}}}$, $V^{\raisebox{0.5pt}{\scalebox{0.6}{$\bullet$}}}$ in $\com^{n}_{\CC'}$ and $\BE'$-extensions $\delta\in\BE'(U^{n+1},A)$, $\eps\in\BE'(V^{n+1},B)$ such that  $\fs'(\delta)=[U^{\raisebox{0.5pt}{\scalebox{0.6}{$\bullet$}}}]$, $\fs'(\eps)=[V^{\raisebox{0.5pt}{\scalebox{0.6}{$\bullet$}}}]$, $a=d_{U}^{0}$ and $b=d_{V}^{0}$. 
Since $\delta$ 
is also an $\BE$-extension, there is a complex $X^{\raisebox{0.5pt}{\scalebox{0.6}{$\bullet$}}}$ 
in $\com^{n}_{\CC}$ with $X^{0}=\SG A$ and $X^{n+1}=\SG U^{n+1}$, 
such that $\fs(\delta)=[X^{\raisebox{0.5pt}{\scalebox{0.6}{$\bullet$}}}]$. 
Thus, we have 
$\fs'(\delta)=
[W^{\raisebox{0.5pt}{\scalebox{0.6}{$\bullet$}}}]
$, 
where $W^{\raisebox{0.5pt}{\scalebox{0.6}{$\bullet$}}}$ 
is the complex 
\[
\begin{tikzcd}[column sep=2.15cm]
A \arrow{r}{(\SF d_{X}^{0})\Phi_{A}}& \SF X^{1} \arrow{r}{\SF d_{X}^{1}}& \cdots \arrow{r}{\SF d_{X}^{n-1}}& \SF X^{n} \arrow{r}{\tensor[]{\Phi}{_{U^{n+1}}^{-1}}\SF d_{X}^{n}}&U^{n+1}.
\end{tikzcd}
\] 
In particular, this implies that the complexes $U^{\raisebox{0.5pt}{\scalebox{0.6}{$\bullet$}}}$ and 
$W^{\raisebox{0.5pt}{\scalebox{0.6}{$\bullet$}}}$
are homotopy equivalent in $\com_{(A,U^{n+1})}^{n}$, 
and in turn implies that 
$\SG_{\com} U^{\raisebox{0.5pt}{\scalebox{0.6}{$\bullet$}}}$ and $\SG_{\com} W^{\raisebox{0.5pt}{\scalebox{0.6}{$\bullet$}}}$ 
are homotopy equivalent in $\com^{n}_{(\SG A,\SG U^{n+1})}$. 

Next we show that the following diagram depicts an isomorphism of complexes $\SG_{\com}W^{\raisebox{0.5pt}{\scalebox{0.6}{$\bullet$}}} \to X^{\raisebox{0.5pt}{\scalebox{0.6}{$\bullet$}}}$.
\begin{equation}\label{eqn:GFX-dot-homotopic-to-X-dot}
\hspace{-9pt}
\begin{tikzcd}[column sep=2.43cm]
\SG A \arrow{r}{(\SG \SF d_{X}^{0})\SG \Phi_{A}}\arrow[equals]{d}& \SG \SF X^{1} \arrow{r}{\SG \SF d_{X}^{1}}\arrow{d}{\Psi_{X^{1}}}& \cdots \arrow{r}{\SG \SF d_{X}^{n-1}}& \SG \SF X^{n} \arrow{r}{\tensor[]{\SG \Phi}{_{U^{n+1}}^{-1}}(\SG \SF d_{X}^{n})}\arrow{d}{\Psi_{X^{n}}}&\SG U^{n+1}\arrow[equals]{d}\\
\SG A \arrow{r}{d_{X}^{0}}& X^{1} \arrow{r}{d_{X}^{1}}& \cdots \arrow{r}{d_{X}^{n-1}}& X^{n} \arrow{r}{d_{X}^{n}}&\SG U^{n+1}.
\end{tikzcd}
\end{equation} 
For each $i=0,\ldots, n$, the square  
\begin{equation}\label{eqn:comm-square-dXi-using-nat-iso-Psi}
\begin{tikzcd}[column sep=1.5cm]
\SG \SF X^{i} \arrow{r}{\SG \SF d_{X}^{i}}\arrow{d}{\Psi_{X^{i}}}& \SG \SF X^{i+1}\arrow{d}{\Psi_{X^{i+1}}} \\
X^{i} \arrow{r}{d_{X}^{i}}& X^{i+1}
\end{tikzcd}
\end{equation}
commutes since $\Psi$ is a natural isomorphism. 
This shows that all except possibly the far left and far right squares in \eqref{eqn:GFX-dot-homotopic-to-X-dot} commute. 
To see the far left square commutes, observe that 
\begin{align*}
\Psi_{X^{1}}(\SG\SF d_{X}^{0}) \SG \Phi_{A} 
	&= d_{X}^{0}\Psi_{\SG A}\SG \Phi_{A} && \text{by } \eqref{eqn:comm-square-dXi-using-nat-iso-Psi}\\
	&= d_{X}^{0}\circ 1_{\SG A}&&\text{using the triangle identities } \eqref{eqn:triangle-identities} \\
	& = d_{X}^{0}.
\end{align*}
Similarly, one can show the far right square of \eqref{eqn:GFX-dot-homotopic-to-X-dot} also commutes, and hence \eqref{eqn:GFX-dot-homotopic-to-X-dot} is an isomorphism of complexes in $\com_{(\SG A, \SG U^{n+1})}^{n}$. 
This implies that 
$\fs(\delta) 
	= [X^{\raisebox{0.5pt}{\scalebox{0.6}{$\bullet$}}}]
	= [\SG_{\com} W^{\raisebox{0.5pt}{\scalebox{0.6}{$\bullet$}}}]
	= [\SG_{\com} U^{\raisebox{0.5pt}{\scalebox{0.6}{$\bullet$}}}]$, 
so $\SG a$ is an $\fs$-inflation.

In the same way, it can be shown that 
$\SG b$ is also an $\fs$-inflation. 
Since $(\CC,\BE,\fs)$ satisfies axiom \ref{nEA1}, we have that the composition $\SG (ba)=\SG b \circ \SG a$ is an $\fs$-inflation. 
Thus, there exist $Z^{n+1}\in\CC$, $\zeta\in\BE(Z^{n+1},\SG A)$ and $Z^{\raisebox{0.5pt}{\scalebox{0.6}{$\bullet$}}}\in\com_{(\SG A,Z^{n+1})}^{n}$, such that 
$\fs(\zeta)=[Z^{\raisebox{0.5pt}{\scalebox{0.6}{$\bullet$}}}]$, $Z^{0}=\SG A$, $Z^{1}=\SG C$ and $d_{Z}^{0}=\SG(ba)$. 
As $\SG$ is essentially surjective, there is an object $T\in\CC'$ and an isomorphism $h\colon \SG T\to Z^{n+1}$. 
Denote by $j\deff h^{-1}\colon Z^{n+1}\to\SG T$ the inverse for $h$. 
Since 
$\zeta 
= (1_{Z^{n+1}})^{\BE}\zeta 
= (hj)^{\BE}\zeta
= j^{\BE}h^{\BE}\zeta$, 
we have 
a sequence 
\[
\begin{tikzcd}[column sep=1.3cm]
\zeta=j^{\BE}(h^{\BE}\zeta)\arrow{r}{(1_{\SG A}, j)} &  h^{\BE}\zeta\arrow{r}{(1_{\SG A}, h)}& \zeta
\end{tikzcd}
\] 
of morphisms of $\BE$-extensions, which may be realised as follows
\[
\begin{tikzcd}[column sep=0.9cm]
\phantom{lll}Z^{\raisebox{0.5pt}{\scalebox{0.6}{$\bullet$}}}\colon \arrow{d}{j^{\raisebox{0.8pt}{\scalebox{0.5}{$\bullet$}}}}& \SG A \arrow{r}{\SG (ba)}\arrow[equals]{d}& \SG C \arrow{r}{d_{Z}^{1}}\arrow{d}{j^{1}}& Z^{2} \arrow{r}{d_{Z}^{2}}\arrow{d}{j^{2}}& \cdots \arrow{r}{d_{Z}^{n-1}}& Z^{n} \arrow{r}{d_{Z}^{n}}\arrow{d}{j^{n}}& Z^{n+1}\arrow{d}{j}[swap]{\iso}\\
\phantom{lll}D^{\raisebox{0.5pt}{\scalebox{0.6}{$\bullet$}}}\colon \arrow{d}{h^{\raisebox{0.8pt}{\scalebox{0.5}{$\bullet$}}}}& \SG A \arrow{r}{}\arrow[equals]{d}& D^{1}\arrow{r}{} \arrow{d}{h^{1}}& D^{2}\arrow{r}{}\arrow{d}{h^{2}} & \cdots \arrow{r}{}& D^{n} \arrow{r}{}\arrow{d}{h^{n}}& \SG T\arrow{d}{h}[swap]{\iso}\\
\phantom{lll}Z^{\raisebox{0.5pt}{\scalebox{0.6}{$\bullet$}}}\colon & \SG A \arrow{r}{\SG (ba)}& \SG C \arrow{r}{d_{Z}^{1}}& Z^{2} \arrow{r}{d_{Z}^{2}}& \cdots \arrow{r}{d_{Z}^{n-1}}& Z^{n} \arrow{r}{d_{Z}^{n}}& Z^{n+1}
\end{tikzcd}
\] 
Note that $\lan D^{\raisebox{0.5pt}{\scalebox{0.6}{$\bullet$}}}, h^{\BE}\zeta \ran$ is an $n$-exangle as $\fs(h^{\BE}\zeta) = [D^{\raisebox{0.5pt}{\scalebox{0.6}{$\bullet$}}}]$. 
Let $E^{\raisebox{0.5pt}{\scalebox{0.6}{$\bullet$}}}$ denote the complex 
\[
\begin{tikzcd}[column sep=0.9cm]
\SG A \arrow{r}{\SG (ba)}& \SG C \arrow{r}{d_{Z}^{1}}& Z^{2} \arrow{r}{d_{Z}^{2}}& \cdots \arrow{r}{d_{Z}^{n-1}}& Z^{n} \arrow{r}{jd_{Z}^{n}}& \SG T
\end{tikzcd}
\] 
in $\com_{(\SG A, \SG T)}^{n}$. 
It is straightforward to verify that 
$\lan E^{\raisebox{0.5pt}{\scalebox{0.6}{$\bullet$}}}, h^{\BE}\zeta\ran$ 
is also an $n$-exangle.
Moreover, we have morphisms 
\[
(1_{\SG A},j^{1},\ldots,j^{n}, 1_{\SG T})\in\com_{(\SG A, \SG T)}^{n}
(E^{\scalebox{0.6}{$\bullet$}},D^{\scalebox{0.6}{$\bullet$}})
\hspace{1cm}
\text{and}
\hspace{1cm}
(1_{\SG A},h^{1},\ldots,h^{n}, 1_{\SG T})\in\com_{(\SG A, \SG T)}^{n}
(D^{\scalebox{0.6}{$\bullet$}},E^{\scalebox{0.6}{$\bullet$}}).
\] 
Thus, by \cite[Prop.\ 2.21]{HerschendLiuNakaoka-n-exangulated-categories-I-definitions-and-fundamental-properties}, we have that $(1_{\SG A},j^{1},\ldots,j^{n}, 1_{\SG T})$ is a homotopy equivalence, and so 
$\fs(h^{\BE}\zeta) 
= [D^{\raisebox{0.5pt}{\scalebox{0.6}{$\bullet$}}}]
= [E^{\raisebox{0.5pt}{\scalebox{0.6}{$\bullet$}}}]$. 
Therefore, we have 
\[
\fs'(h^{\BE}\zeta)= [
\begin{tikzcd}[column sep=1.65cm]
A \arrow{r}{(\SF\SG ba)\Phi_{A}}& \SF\SG C \arrow{r}{\SF d_{Z}^{1}}& \SF Z^{2} \arrow{r}{\SF d_{Z}^{2}}& \cdots \arrow{r}{\SF d_{Z}^{n-1}}& \SF Z^{n} \arrow{r}{d}& T
\end{tikzcd}],
\] 
where $d\deff\tensor[]{\Phi}{_{T}^{-1}}\SF (jd_{Z}^{n})$. 
Since there is an isomorphism 
\[
\begin{tikzcd}[column sep=1.65cm, row sep=1cm]
A \arrow{r}{(\SF\SG ba)\Phi_{A}}\arrow[equals]{d}& \SF\SG C \arrow{r}{\SF d_{Z}^{1}}\arrow{d}{\tensor[]{\Phi}{_{C}^{-1}}}& \SF Z^{2} \arrow{r}{\SF d_{Z}^{2}}\arrow[equals]{d}& \cdots \arrow{r}{\SF d_{Z}^{n-1}}& \SF Z^{n} \arrow{r}{d}\arrow[equals]{d}& T\arrow[equals]{d}\\
A \arrow{r}{ba}& C \arrow{r}{(\SF d_{Z}^{1})\Phi_{C}}& \SF Z^{2} \arrow{r}{\SF d_{Z}^{2}}& \cdots \arrow{r}{\SF d_{Z}^{n-1}}& \SF Z^{n} \arrow{r}{d}& T
\end{tikzcd}
\]
of complexes in $\com_{(A,T)}^{n}$, we see that $ba$ is indeed an $\fs'$-inflation. 
Dually, $\fs'$-deflations are closed under composition, and $(\CC',\BE',\fs')$ satisfies \ref{nEA1}.

Lastly, we show \ref{nEA2} holds for $(\CC',\BE',\fs')$; ($n$-EA$2^{\op}$) can be shown using a similar argument. 
Thus, let $\delta \in \BE'(D,A)$ and $c\colon C\to D$ be arbitrary, where $A,C,D\in\CC'$. 
Then we may realise 
$c^{\BE'}\delta=(\SG c)^{\BE}\delta\in\BE(\SG C,\SG A)$ 
by a complex
\[
\begin{tikzcd}
\SG A \arrow{r}{d_{X}^{0}} & \SG B\arrow{r}{d_{X}^{1}} & X^{2} \arrow{r}{d_{X}^{2}}& \cdots\arrow{r}{d_{X}^{n-1}} & X^{n} \arrow{r}{d_{X}^{n}}& \SG C
\end{tikzcd}
\] 
in $\com_{\CC}^{n}$ for some $B\in \CC'$, using that $\SG$ essentially surjective. 
Since $\SG$ is also full, we have that $d_{X}^{0}=\SG a$ for some $a\colon A\to B$. 
Furthermore, $\delta\in\BE(\SG D,\SG A)$ so we may realise this $\BE$-extension by a complex 
\[
\begin{tikzcd}
\SG A \arrow{r}{d_{Y}^{0}}& Y^{1}\arrow{r}{d_{Y}^{1}} &Y^{2}\arrow{r}{d_{Y}^{2}}& \cdots\arrow{r}{d_{Y}^{n-1}} & Y^{n} \arrow{r}{d_{Y}^{n}}& \SG D
\end{tikzcd}
\] 
in $\com^{n}_{\CC}$. 
By \ref{nEA2} for the $n$-exangulated category $(\CC,\BE,\fs)$, we may realise the morphism 
$(1_{\SG A},\SG c)\colon(\SG c)^{\BE}\delta\to \delta$ 
of $\BE$-extensions by a good lift $f^{\raisebox{0.5pt}{\scalebox{0.6}{$\bullet$}}}=(1_{\SG A},f^{1},f^{2},\ldots,f^{n},\SG c)$. 
That is, $f^{\raisebox{0.5pt}{\scalebox{0.6}{$\bullet$}}}$ is a morphism 
\[
\begin{tikzcd}[column sep=1.2cm, row sep=1cm]
\SG A \arrow{r}{d_{X}^{0}} \arrow[equals]{d}& \SG B\arrow{r}{d_{X}^{1}} \arrow{d}{f^{1}}& X^{2} \arrow{r}{d_{X}^{2}}\arrow{d}{f^{2}}& \cdots\arrow{r}{d_{X}^{n-1}} & X^{n} \arrow{r}{d_{X}^{n}}\arrow{d}{f^{n}}& \SG C\arrow{d}{\SG c}\\
\SG A \arrow{r}{d_{Y}^{0}}& Y^{1}\arrow{r}{d_{Y}^{1}} &Y^{2}\arrow{r}{d_{Y}^{2}} & \cdots\arrow{r}{d_{Y}^{n-1}} & Y^{n} \arrow{r}{d_{Y}^{n}}& \SG D
\end{tikzcd}
\] 
of complexes such that 
$\fs((d_{X}^{0})_{\BE}\delta)=[M^{\raisebox{0.5pt}{\scalebox{0.6}{$\bullet$}}}]$, 
where $M^{\raisebox{0.5pt}{\scalebox{0.6}{$\bullet$}}}\deff\cone(\wh{f\hspace{1.5pt}})^{\raisebox{0.5pt}{\scalebox{0.6}{$\bullet$}}}$ is the mapping cone 
\[
\begin{tikzcd}
\SG B \arrow{r}{d_{M}^{-1}}& X^{2}\oplus Y^{1} \arrow{r}{d_{M}^{0}}& \cdots \arrow{r}{}& X^{n}\oplus Y^{n-1} \arrow{r}{d_{M}^{n-2}}& \SG C\oplus Y^{n} \arrow{r}{d_{M}^{n-1}}& \SG D
\end{tikzcd}
\]
of $\wh{f}=(f^{1},\ldots,f^{n},\SG c)$. 
As in Definition \ref{def:mapping-cone-as-in-Jasso}, the differentials are  
$d_{M}^{-1}\deff \begin{psmallmatrix}-d_{X}^{1}\\f^{1} \end{psmallmatrix}$, 
$d_{M}^{n-1}\deff(\,\SG c\;\;\,  d_{Y}^{n}\,)$, 
and 
\[
d_{M}^{i}\deff
\begin{pmatrix}-d_{X}^{i+2} & 0\\f^{i+2} & d_{Y}^{i+1} \end{pmatrix}
\]
for $i\in\{0,\ldots,n-2\}$. 
Then we see that the morphism $(1_{A},c)\colon c^{\BE'}\delta \to \delta$ of $\BE'$-extensions is realised by the morphism 
\[
\begin{tikzcd}[column sep=1.6cm, row sep=1cm]
A \arrow{r}{(\SF d_{X}^{0})\Phi_{A}} \arrow[equals]{d}& \SF \SG B\arrow{r}{\SF d_{X}^{1}} \arrow{d}{\SF f^{1}}& \SF X^{2} \arrow{r}{\SF d_{X}^{2}}\arrow{d}{\SF f^{2}}& \cdots\arrow{r}{\SF d_{X}^{n-1}} & \SF X^{n} \arrow{r}{\tensor[]{\Phi}{_{C}^{-1}}\SF d_{X}^{n}}\arrow{d}{\SF f^{n}}& C\arrow{d}{c}\\
A \arrow{r}{(\SF d_{Y}^{0})\Phi_{A}}& \SF Y^{1}\arrow{r}{\SF d_{Y}^{1}} &\SF Y^{2}\arrow{r}{\SF d_{Y}^{2}} & \cdots\arrow{r}{\SF d_{Y}^{n-1}} & \SF Y^{n} \arrow{r}{\tensor[]{\Phi}{_{D}^{-1}}\SF d_{Y}^{n}}& D
\end{tikzcd}
\] 
of complexes in $\CC'$ 
since $\Phi\colon \idfunc{\CC'}\Rightarrow \SF \SG$ is a natural isomorphism. 
Using $\Phi$ again and that $\SG a=d_{X}^{0}$, we can adjust the diagram above appropriately to obtain a morphism $g^{\raisebox{0.5pt}{\scalebox{0.6}{$\bullet$}}}\deff(1_{A},(\SF f^{1})\Phi_{B},\SF f^{2},\ldots, \SF f^{n}, c)$
\[
\begin{tikzcd}[column sep=1.6cm, row sep=1cm]
A \arrow{r}{a} \arrow[equals]{d}& B\arrow{r}{(\SF d_{X}^{1})\Phi_{B}} \arrow{d}{(\SF f^{1})\Phi_{B}}& \SF X^{2} \arrow{r}{\SF d_{X}^{2}}\arrow{d}{\SF f^{2}}& \cdots\arrow{r}{\SF d_{X}^{n-1}} & \SF X^{n} \arrow{r}{\tensor[]{\Phi}{_{C}^{-1}}\SF d_{X}^{n}}\arrow{d}{\SF f^{n}}& C\arrow{d}{c}\\
A \arrow{r}{(\SF d_{Y}^{0})\Phi_{A}}& \SF Y^{1}\arrow{r}{\SF d_{Y}^{1}} &\SF Y^{2}\arrow{r}{\SF d_{Y}^{2}} & \cdots\arrow{r}{\SF d_{Y}^{n-1}} & \SF Y^{n} \arrow{r}{\tensor[]{\Phi}{_{D}^{-1}}\SF d_{Y}^{n}}& D,
\end{tikzcd}
\] 
giving another realisation of $(1_{A},c)\colon c^{\BE'}\delta \to \delta$.

We claim that $\fs'(a_{\BE'}\delta)=[\cone(\wh{g\hspace{0.5pt}})^{\raisebox{0.5pt}{\scalebox{0.6}{$\bullet$}}}]$, where $\wh{g\hspace{1pt}}^{\raisebox{0.5pt}{\scalebox{0.6}{$\bullet$}}}\deff((\SF f^{1})\Phi_{B},\SF f^{2},\ldots, \SF f^{n}, c)$. 
Since $\fs((d_{X}^{0})_{\BE}\delta)=[M^{\raisebox{0.5pt}{\scalebox{0.6}{$\bullet$}}}]$ and $\SG a=d_{X}^{0}$, we have that 
$
\fs'(a_{\BE'}\delta)=\fs'((\SG a)_{\BE}\delta)
=\fs'((d_{X}^{0})_{\BE}\delta)$
is the homotopy equivalence class
\[
[\begin{tikzcd}[column sep=1.35cm]
B \arrow{r}{(\SF d_{M}^{-1})\Phi_{B}}& \SF X^{2}\oplus \SF Y^{1} \arrow{r}{\SF d_{M}^{0}}& \cdots \arrow{r}{\SF d_{M}^{n-3}}& \SF X^{n}\oplus \SF Y^{n-1} \arrow{r}{\SF d_{M}^{n-2}}& \SF \SG C\oplus \SF Y^{n} \arrow{r}{\tensor[]{\Phi}{_{D}^{-1}}\SF d_{M}^{n-1}}& D
\end{tikzcd}].
\] 
Note that there is an isomorphism
\[
\hspace*{-3pt}
\begin{tikzcd}[column sep=1.45cm, row sep=1.3cm, ampersand replacement=\&]
B \arrow{r}{(\SF d_{M}^{-1})\Phi_{B}} \arrow[equals]{d}\& \SF X^{2}\oplus \SF Y^{1} \arrow{r}{\SF d_{M}^{0}} \arrow[equals]{d}\& \cdots \arrow{r}{\SF d_{M}^{n-3}} \& \SF X^{n}\oplus \SF Y^{n-1} \arrow{r}{\SF d_{M}^{n-2}} \arrow[equals]{d}\& C\oplus \SF Y^{n} \arrow{r}{e} \arrow{d}{\begin{psmallmatrix}
\Phi_{C} & 0 \\ 0 & 1_{\SF Y^{n}}
\end{psmallmatrix}}[swap]{\iso}\& D \arrow[equals]{d}\\
B \arrow{r}{(\SF d_{M}^{-1})\Phi_{B}} \& \SF X^{2}\oplus \SF Y^{1} \arrow{r}{\SF d_{M}^{0}} \& \cdots \arrow{r}{\SF d_{M}^{n-3}} \& \SF X^{n}\oplus \SF Y^{n-1} \arrow{r}{\SF d_{M}^{n-2}} \& \SF \SG C\oplus \SF Y^{n} \arrow{r}{\tensor[]{\Phi}{_{D}^{-1}}\SF d_{M}^{n-1}} \& D
\end{tikzcd}
\] 
of complexes, where $e\deff(\,
c \;\;\;\, \tensor[]{\Phi}{_{D}^{-1}}\SF d_{Y}^{n}\,)$, in which the top row is precisely $\cone(\wh{g\hspace{0.5pt}})^{\raisebox{0.5pt}{\scalebox{0.6}{$\bullet$}}}$. 
Hence, we see that the $\BE'$-extension $a_{\BE'}\delta$ is realised by $\cone(\wh{g\hspace{0.5pt}})^{\raisebox{0.5pt}{\scalebox{0.6}{$\bullet$}}}$, and $g^{\raisebox{0.5pt}{\scalebox{0.6}{$\bullet$}}}$ is a good lift of $(1_{A},c)$. 
Thus, \ref{nEA2} holds for $(\CC',\BE',\fs')$ and we are done.

\end{proof}

\begin{rem}

In Theorem \ref{thm:transport-of-n-exangulated-structure}, the $n$-exangulated structure possessed by a category $\CC$ is transported across an equivalence $\SF\colon \CC\to\CC'$, so that $\CC'$ becomes an $n$-exangulated category and $\SF$ becomes an $n$-exangle equivalence. 
We remark here that the $n$-exangulated structure acquired by $\CC'$ is unique up to $n$-exangle equivalence. 
Analogous statements also hold for the settings of Theorems \ref{thm:transport-of-n-angulated-structure}, \ref{thm:transport-of-n-exact-structure} and  \ref{thm:transport-of-n-abelian-structure}.

\end{rem}

\begin{rem}

Suppose $\SF\colon\CC\to\CC'$ and $\SG\colon\CC'\to\CC$ are functors such that $(\SF,\SG)$ and $(\SG,\SF)$ are both adjoint pairs. 
Recall that this holds whenever $\SF$ is an equivalence and $\SG$ is a quasi-inverse for $\SF$, but that the converse is false in general\footnote{We refer the reader to an example provided by user `Hanno' on an online forum; see \url{https://math.stackexchange.com/q/1095196}.}. 
We note here that our proof of Theorem \ref{thm:transport-of-n-exangulated-structure} cannot be directly generalised to this setting: 
we rely on the assumption that $\SF$ and $\SG$ are equivalences, so that the unit and counit are both natural isomorphisms. 
This is needed in Lemma \ref{lem:transport-of-exact-realisation} to check that the assignment $\fs'$ is well-defined, as well as in both Lemma \ref{lem:transport-of-exact-realisation} and Theorem \ref{thm:transport-of-n-exangulated-structure} to produce several isomorphisms between complexes.

\end{rem}


\medskip
\section{Special case: skeletal categories}
\label{sec:special-case-skeletal-categories}

Suppose $\CC'$ is a skeletal category equivalent to a weak (pre-)$(n+2)$-angulated category. Then we know $\CC'$ is also weak (pre-)$(n+2)$-angulated by Theorem \ref{thm:transport-of-n-angulated-structure}. The goal of \S\ref{sec:special-case-skeletal-categories} is to show that the induced $(n+2)$-suspension functor of $\CC'$ is actually an \emph{automorphism}, not just an autoequivalence, thereby showing $\CC'$ is what we call a strong (pre-)$(n+2)$-angulated category.

First, let us recall what it means for a category to be skeletal and how a skeleton of a category is defined.

\begin{defn}
\label{def:skeletal-category}

\cite[\S{}IV.4]{MacLane-categories-for-the-working-mathematician} 
A category $\CA$ is said to be \emph{skeletal} if $X\neq Y$ implies $X$ is not isomorphic to $Y$ in $\CA$ for all $X,Y\in\CA$.

\end{defn}

\begin{defn}
\label{def:skeleton-of-category}

\cite[\S{}IV.4]{MacLane-categories-for-the-working-mathematician} 
Let $\CA$ be a category. A \emph{skeleton} of $\CA$ is a full, skeletal subcategory $\CB$ of $\CA$ such that the (fully faithful) canonical inclusion $\SF\colon \CB \into \CA$ is essentially surjective.

\end{defn}

\begin{rem}

Note that $\CA$ may have many skeletons, but they are all isomorphic; see \cite[Exer. IV.4.1]{MacLane-categories-for-the-working-mathematician}. 

\end{rem}

The following proposition is a key tool in the proof of the main result of \S\ref{sec:special-case-skeletal-categories}. We include the details for the convenience of the reader.

\begin{prop}
\label{prop:equivalence-between-skeletal-categories-is-an-isomorphism}

\cite[\S{}I, 1.364]{FreydScedrov-categories-allegories}\footnote{Although this result is stated in \cite{FreydScedrov-categories-allegories}, we would like to acknowledge that user `Eric Wofsey' provides the idea behind our proof on an online forum; see \url{https://math.stackexchange.com/q/1658991}.}
Suppose $\CA$ and $\CB$ are both skeletal categories and that $\SF\colon \CA\to \CB$ is an equivalence of categories. Then $\SF$ is an isomorphism of categories.

\end{prop}

\begin{proof}

We need to define a functor $\SG\colon \CB \to \CA$ such that $\SG\SF = \idfunc{\CA}$ and $\SF\SG = \idfunc{\CB}$. In order to do this, we first show that $\SF$ is bijective on objects. 
Let $X\in\CB$ be arbitrary. As $\SF$ is an equivalence, $\SF$ is essentially surjective and so there exists $A\in\CA$ such that $\SF A\iso X$. Since $\CB$ is skeletal, we must have $\SF A = X$. Thus, $\SF$ is surjective on objects. 
Now suppose we have $A,B\in \CA$ such that 
$\SF A = X= \SF B$ in $\CB$. Consider the identity morphism 
$1_{X}\in\CB(\SF A, \SF B) = \CB(\SF B, \SF A).$ 
Since $\SF$ is full, there exist morphisms $f\colon A\to B$ 
and $g\colon B\to A$ such that $\SF f = 1_{X}$ and $\SF g = 1_{X}$. Then $\SF(gf) = 1_{X} = \SF(1_{A})$ and $\SF(fg) = 1_{X} = \SF(1_{B})$, whence $gf = 1_{A}$ and $fg = 1_{B}$ as $\SF$ is faithful. Thus, $A\iso B$ and we have $A=B$ as $\CA$ is skeletal. Hence, $\SF$ is injective on objects also. 

Now define $\SG\colon \CB \to \CA$ as follows. For an object $X\in\CB$, there is precisely one object $A\in\CA$ such that $\SF A = X$ (as $\SF$ is bijective on objects). In this case, put $\SG(X) = A$. Let $f\colon X\to Y$ be an arbitrary morphism in $\CB$. Then $X = \SF A$ and $Y = \SF B$ for some uniquely determined $A,B\in \CA$, so $f\in\CB(\SF A, \SF B)$. As $\SF$ is fully faithful, there is a unique morphism $a\in\CA(A,B)$ such that $\SF a = f$. Thus, set $\SG(f) = a$.

Let us show that $\SG$ is indeed a functor. Let $X\in \CB$ be arbitrary and consider the identity morphism $1_{X}$.  
Then $X=\SF A$ for some $A\in\CA$ and $1_{X} =\SF(1_{A})$, so $\SG(1_{X}) = 1_{A} = 1_{\SG X}$ is the identity morphism. Now suppose we have morphisms $f\colon X\to Y$ and $g\colon Y\to Z$ in $\CB$. 
Then $f = \SF a\colon X= \SF A \to \SF B = Y$ and $g = \SF b\colon Y=\SF B \to \SF C = Z$, where $A,B,C$ and $a,b$ are all uniquely determined. Note that $gf = (\SF b)(\SF a) = \SF(ba)$, so $\SG(gf) = ba = (\SG g)(\SG f)$ and $\SG$ respects composition.

Lastly, let us show that $\SF$ and $\SG$ are mutually inverse functors. Let $A\in\CA$ be arbitrary. 
By definition of $\SG$, the object $\SG(\SF A)$ in $\CA$ is the object $A'\in\CA$ such that $\SF A' = \SF A$. 
Since $\SF$ is injective on objects, we must have $A'=A$ and $\SG(\SF A) = A$. 
Now let $a\colon A\to B$ be an arbitrary morphism in $\CA$, and consider $\SF a\colon \SF A\to \SF B$. 
Then $\SG(\SF a)$ is the unique morphism $a'\colon A\to B$ such that $\SF a'= \SF a$. Using that $\SF$ is faithful, we obtain $a'=a$ and $\SG (\SF a) = a$. 
Hence, we have shown $\SG\SF = \idfunc{\CA}$. 
Similarly, one can also show $\SF\SG = \idfunc{\CB}$.

\end{proof}

We are now in a position to improve Theorem \ref{thm:transport-of-n-angulated-structure} when transporting a weak (pre-)$(n+2)$-angulated structure to a skeletal category.

\begin{thm}
\label{thm:transport-of-n-angulated-structure-to-skeletal-category}

Let $(\CC, \sus, \CT)$ be a weak (pre-)$(n+2)$-angulated category and let $\CC'$ be any skeletal category. Suppose $\SF\colon \CC \to \CC'$ is an equivalence. Then $\CC'$ is a strong (pre-)$(n+2)$-angulated category 
and $\SF$ is an $(n+2)$-angle equivalence.

\end{thm}

\begin{proof}

An application of Theorem \ref{thm:transport-of-n-angulated-structure} shows that $\CC'$ has the structure of a weak $(n+2)$-angulated category, for some $(n+2)$-suspension functor $\sus'\colon \CC'\to \CC'$ that is an autoequivalence, and that $\SF$ is an $(n+2)$-angle equivalence. Since $\sus'$ is an equivalence between skeletal categories, we have that $\sus'$ is actually an isomorphism of categories by Proposition \ref{prop:equivalence-between-skeletal-categories-is-an-isomorphism}. 
That is, the $(n+2)$-suspension functor $\sus '$ of $\CC'$ is an automorphism and $\CC'$ is a strong (pre-)$(n+2)$-angulated category. 

\end{proof}

As an immediate consequence, we have the following.

\begin{cor}
\label{cor:transport-of-n-angulated-structure-to-skeleton}

Let $(\CC, \sus, \CT)$ be a weak (pre-)$(n+2)$-angulated category. Then any skeleton $\CC^{\skel}$ of $\CC$ is a strong (pre-)$(n+2)$-angulated category.

\end{cor}

\begin{proof}

Let $\SF\colon \CC^{\skel}\into \CC$ be the inclusion functor associated to a skeleton $\CC^{\skel}$ of $\CC$. Then $\SF$ is an equivalence of categories, and we may apply Theorem \ref{thm:transport-of-n-angulated-structure-to-skeletal-category} with $\CC'=\CC^{\skel}$.

\end{proof}

We conclude \S\ref{sec:special-case-skeletal-categories} by clarifying some details involved in the construction of the cluster category, a certain orbit category which was first defined in \cite{BMRRT-cluster-combinatorics}. Let $\CA$ be an additive category and suppose $\SG$ is an automorphism of $\CA$. 
The \emph{orbit category} $\CA/\SG$ has the same objects as $\CA$, and
for $X,Y\in\CA/\SG$ we set 
\[
(\CA/\SG)(X,Y)\deff\underset{i\in\BZ}{\coprod}\CA(\SG^{i}X,Y);
\] 
see \cite[Def.\ 2.3]{CibilsMarcos-skew-category-galois-covering-smash-product-of-k-category}, \cite[\S{}1]{Keller-triangulated-orbit-categories}. 
Note that in order to have well-defined composition of morphisms in $\CA/\SG$, we need $\SG$ to be an automorphism.

\begin{example}

Let $k$ be a field and let $\Lambda$ be a finite-dimensional, hereditary $k$-algebra. 
Let $\der \deff \der^{b}(\lmod{\Lambda})$ denote the bounded derived category of finite-dimensional left $\Lambda$-modules. 
Then $\der$ is a triangulated (=($1+2$)-angulated) category (see \cite[Prop.\ 3.5.40]{Zimmermann-rep-theory-book}) and denote its suspension functor by $[1]$. 
In addition, $\der$ has Auslander-Reiten triangles (see \cite[Thm.\ I.4.6]{Happel-triangulated-cats-in-rep-theory}), and let $\tau^{-1}$ denote a quasi-inverse for the Auslander-Reiten translation (see \cite[p.\ 42]{Happel-triangulated-cats-in-rep-theory}). 
Note that $\tau^{-1}$ is an autoequivalence of $\der$, so the composite endofunctor $\SF\deff\tau^{-1}\circ[1]$ is also an autoequivalence of $\der$.

The \emph{cluster category} of $\Lambda$ is defined to be the orbit category $\CC_{\Lambda}\deff\der/\SF$ (see \cite[p.\ 576]{BMRRT-cluster-combinatorics}). 
Recall that for an orbit category $\CA/\SG$, the functor $\SG$ is required be to an automorphism. 
However, $\SF$ is only an autoequivalence, and not an automorphism, of $\der$. 
Thus, the definition of $\CC_{\Lambda}$ just stated masks some details that we now discuss. 
We first take a skeleton $\der^{\skel}$ of $\der$. 
Then $\SF$ induces an autoequivalence $\SF'$ of $\der^{\skel}$ and, by Proposition \ref{prop:equivalence-between-skeletal-categories-is-an-isomorphism}, $\SF'$ is actually an automorphism. 
(This is the `standard construction' alluded to in \cite[\S1]{Keller-triangulated-orbit-categories}.) 
Hence, we may consider the cluster category of $\Lambda$ to be the orbit category $\der^{\skel}/\SF'$. 
Moreover, $\der^{\skel}$ is triangulated by Corollary \ref{cor:transport-of-n-angulated-structure-to-skeleton}, and we may thus apply \cite[Thm.\ 1]{Keller-triangulated-orbit-categories} to conclude $\der^{\skel}/\SF'$ is triangulated.

\end{example}


\begin{acknowledgements}

Both authors would like to thank Peter J\o{}rgensen, Bernhard Keller and Bethany R.\ Marsh for helpful discussions and for their advice during the preparation of this article. We also thank the referee for reading a previous version of the article and their comments. 

During part of this work the first author was supported by the Alexander von Humboldt Foundation in the framework of an Alexander von Humboldt Professorship endowed by the German Federal Ministry of Education and Research, for which they are grateful. The second author gratefully acknowledges support from the EPSRC (grant EP/P016014/1), and from the London Mathematical Society for funding through an Early Career Fellowship with support from Heilbronn Institute for Mathematical Research. Some of this work was carried out as part of the second author's Ph.D.\ at the University of Leeds. 

\end{acknowledgements}


\bibliographystyle{mybst}
\bibliography{references}
\end{document}